\documentclass[11pt]{amsart}
\usepackage{amsfonts,amssymb,amsmath,amsthm}
\usepackage{url}
\usepackage{enumerate}
\usepackage{hyperref}
\usepackage{fancyhdr}
\newtheorem{theorem}{Theorem}[section]
\newtheorem{lemma}[theorem]{Lemma}
\newtheorem{conjecture}[theorem]{Conjecture}

\newtheorem{corollary}[theorem]{Corollary}
\newtheorem{example}[theorem]{Example}

\theoremstyle{definition}
\newtheorem{definition}[theorem]{Definition}

\newtheorem{remark}[theorem]{Remark}

\numberwithin{equation}{section}
\def\notdiv{\nmid}

\def\lien{\mathrel{\mkern-4mu}}
\def\too{\relbar\lien\rightarrow}
\def\tooo{\relbar\lien\relbar\lien\too}

\def\Q{\mathbb{Q}}
\def\Z{\mathbb{Z}}
\def\F{\mathbb{F}}

\def\Cl{{\mathcal C}\hskip-2pt{\ell}}

\def\Frac#1#2{\hbox{\footnotesize $\displaystyle \frac{#1}{#2}$}}
\def\plus{\displaystyle\mathop{\raise 2.0pt \hbox{$\bigoplus $}}\limits}
\def\prd{\displaystyle\mathop{\raise 2.0pt \hbox{$\prod$}}\limits}
\def\sm{\displaystyle\mathop{\raise 2.0pt \hbox{$\sum$}}\limits}
\def\ov{\overline}
\def\wt{\widetilde}
\def\es{\emptyset}
\def\ds{\displaystyle}
\def\cl{c\hskip-1pt{\ell}}

\def\order{\raise1.5pt \hbox{${\scriptscriptstyle \#}$}}

\author[Georges Gras]{Georges Gras}
\address{Villa la Gardette \\ chemin Ch\^ateau Gagni\`ere \\ F--38520 Le Bourg d'Oisans.}
\email{g.mn.gras@wanadoo.fr}



\begin{document}
 
\title[Incomplete abelian $p$-ramification] {Practice of incomplete 
$p$-ramification \\ over a number field \\ \vspace{0.2cm}
{\scriptsize Appendix}: \\
History of abelian $p$-ramification}

\date{June 25, 2019}

\begin{abstract} {The theory of $p$-ramification, regarding the Galois group of 
the maximal pro-$p$-extension of a number field $K$, unramified outside $p$ 
and $\infty$, is well known including numerical experiments with PARI/GP 
programs. The case of ``incomplete $p$-ramification'' (i.e., when the set $S$ 
of ramified places is a strict subset of the set $P$ of the $p$-places) 
is, on the contrary, mostly unknown in a theoretical point of view.
We give, in a first part, a way to compute, for any $S \subseteq P$, the 
structure of the Galois group of the maximal $S$-ramified abelian pro-$p$-extension 
$H_{K,S}$ of any field $K$ given by means of an irreducible polynomial. 
We publish PARI/GP programs usable without any special prerequisites. 
Then, in an Appendix, we recall the ``story'' of abelian $S$-ramification 
restricting ourselves to elementary aspects in order to precise much basic 
contributions and references, often disregarded, which may be used by specialists
of other domains of number theory. Indeed, the torsion ${\mathcal T}_{K,S}$ of 
${\rm Gal}(H_{K,S}/K)$ (even if $S=P$) is a fundamental obstruction in many 
problems. All relationships involving $S$-ramification, as Iwasawa's theory, Galois 
cohomology, $p$-adic $L$-functions, elliptic curves, algebraic geometry, 
would merit special developments, which is not the purpose of this text.}
\end{abstract}

\maketitle

\tableofcontents

\section{Introduction and basic results}\label{section1}
Because of the numerous references needed, we cite each of them with author's initials and 
publication year, so that the reader has a chronological view of the contributions.
Many results have been collected in our book (Ed. 2005) quoted \cite{Gr3}\,[Gr2003].

\subsection{Notion of Galois $S$-ramification}
\smallskip
Let $p \geq 2$ be a prime number and let $K$ be a number field; we denote by
$P:= \{{\mathfrak p} \hbox{ prime},\ \,{\mathfrak p}\! \mid\! p\}$ the set of $p$-places 
of $K$ and by $S$ an arbitrary set of finite places (later we shall assume $S \subseteq P$). 

\smallskip
A main problem in Galois theory above $K$ is to study the Galois group 
${\mathcal G}_{K,S}$ of the maximal pro-$p$-extension of $K$ which is
$S$-ramified in the ordinary sense (i.e., unramified outside $S$ and 
non-complexified (= totally split) at the real infinite places of $K$ when $p=2$). 

\smallskip
As we will recall it in detail, in Section \ref{story}, the study of ${\mathcal G}_{K,S}$ 
goes back to fundamental contributions of Serre \cite{Se1}\,[Ser1964], 
\v Safarevi\v c \cite{Sa}\,[Sha1964], Brumer \cite{Br}\,[Bru1966], and has been largely extended, 
from the 1980's, in much works considering $S$-ramification (eventually 
with decomposition of another set $\Sigma$ of finite and infinite places). 

\smallskip
The analogous theory for a local base field has also a long history that we 
shall not consider in this article.

\subsection{Main cohomological invariants}
\smallskip
For complete current information about  the ``cohomology of number fields'', 
see the book of Neukirch--Schmidt--Wingberg \cite[Chapter X]{NSW}\,[NSW2000].

\smallskip
When $S=P$, the $\F_p$-dimension of ${\rm H}^1({\mathcal G}_{K,P},\Z/p\Z)$, 
which gives the minimal number of generators of ${\mathcal G}_{K,P}$, is the 
$p$-rank\,\footnote{\,As usual, the $p$-rank of an abelian group $A$ is the 
$\F_p$-dimension of $A/A^p$.} of the abelianization:
$${\mathcal A}_{K,P} := {\mathcal G}_{K,P}^{\rm ab}:=
{\mathcal G}_{K,P}/[{\mathcal G}_{K,P},{\mathcal G}_{K,P}].$$

Denote by $(r_1, r_2)$ the signature of $K$ (whence $r_1+2\, r_2=[K : \Q]$);
then, the $\F_p$-dimension of ${\rm H}^2({\mathcal G}_{K,P}, \Z/p\Z)$, 
which gives the minimal number of relations between these generators,
fulfills the identity:
$${\rm rk}_p({\rm H}^1({\mathcal G}_{K,P},\Z/p\Z)) = 
{\rm rk}_p({\rm H}^2({\mathcal G}_{K,P},\Z/p\Z)) + r_2+1, $$

\noindent
giving, for the torsion group ${\mathcal T}_{K,P}$ of ${\mathcal A}_{K,P}$
{\it under Leopoldt's conjecture}:
$${\rm rk}_p({\mathcal T}_{K,P})\! =\! {\rm rk}_p({\rm H}^2({\mathcal G}_{K,P},\Z/p\Z)). $$

\subsection{Class field theory}
\smallskip
In the general case for $S$ (possibly containing tame places and 
not all the $p$-places) we may write:
\begin{equation}\label{AT}
{\mathcal A}_{K,S} = \Gamma_{K,S} \plus  {\mathcal T}_{K,S}, 
\ \hbox{with}\  \Gamma_{K,S} \simeq \Z_p^{\wt r_{K,S}^{}}, 
\end{equation}

\noindent
where ${\mathcal T}_{K,S} := {\rm tor}_{\Z_p}({\mathcal A}_{K,S})$
and $\wt r_{K,S}^{} \geq 0$. 

Without any $p$-adic assumption on the group of global units of $K$, 
we still have ${\rm rk}_p({\rm H}^1({\mathcal G}_{K,S},\Z/p\Z)) 
= {\rm rk}_p({\mathcal A}_{K,S})$, but $\wt r_{K,S}^{}$ 
(called the $\Z_p$-rank of ${\mathcal A}_{K,S}$)
is more difficult when $S \subsetneq P$; however,
${\rm rk}_p({\mathcal A}_{K,S}) = \wt r_{K,S}^{} + 
{\rm rk}_p({\mathcal T}_{K,S})$ is computable in 
complete generality with the invariants of class field 
theory for $K$ as follows (\v Safarevi\v c formula):

\smallskip
Let $K_{(S)}^\times$ be the subgroup of $K^\times$ of elements 
prime to $S$ and for any ${\mathfrak p} \in S$, let $K_{\mathfrak p}$ be
the completion of $K$ at ${\mathfrak p}$; then:
\begin{equation}\label{cha}
{\rm rk}_p({\mathcal A}_{K,S})  = {\rm rk}_p \big(V_{K,S}/K_{(S)}^{\times p} \big) 
 + \sm_{{\mathfrak p} \in S \, \cap\, P} [K_{\mathfrak p}  : \Q_p] 
+ \sm_{{\mathfrak p} \in S} \delta_{\mathfrak p} - \delta_K - ( r_1 + r_2 -1) , 
\end{equation}

\noindent
where $V_{K,S} := \big \{ \alpha \in K_{(S)}^{\times},\  (\alpha) = {\mathfrak a}^p 
\  \hbox{for an ideal ${\mathfrak a}$ of $K$}\big\}$,
$\delta_{\mathfrak p} = 1$ or $0$ accor\-ding as $K_{\mathfrak p}$ contains
$\mu_p$ or not, and $\delta_K = 1$ or $0$ according as $K$ contains $\mu_p$ or not.
Thus, from the relation  \eqref{AT}:
\begin{equation}\label{chaT}
{\rm rk}_p({\mathcal T_{K,S}})  = {\rm rk}_p({\mathcal A}_{K,S})  - \wt r_{K,S}^{} 
 =  {\rm rk}_p \big (V_{K,S}/K_{(S)}^{\times p} \big ) 
 + \Big[ \sm_{{\mathfrak p} \in S  \, \cap\,  P} [K_{\mathfrak p}  : \Q_p]  - \wt r_{K,S}^{}\Big]
+ \sm_{{\mathfrak p} \in S} \delta_{\mathfrak p} - \delta_K - ( r_1 + r_2 -1),
\end{equation}

\noindent
where $\wt r_{K,S}^{}$ fulfills the following formula:
\begin{equation}\label{rank}
 \sm_{{\mathfrak p} \in S  \, \cap\,  P} [K_{\mathfrak p}  : \Q_p] 
- \wt r_{K,S}^{} =  {\rm dim}_{\Q_p} \big( \Q_p {\rm log}_{S  \, \cap\,  P}(E_K) \big), 
\end{equation}

\noindent
where $E_K$ is the group of global units of $K$ and
${\rm log}_{S  \, \cap\,  P} := \big({\rm log}_{\mathfrak p} \big)_{{\mathfrak p} \in S  \, \cap\, P}$
the family of $p$-adic logarithms over $S  \, \cap\, P$ with values in 
$\bigoplus_{{\mathfrak p} \in S \,\cap\, P} K_{\mathfrak p}$. Note that for $S=P$,
$r_{K,P}^{} := {\rm dim}_{\Q_p} \big( \Q_p {\rm log}_{P}(E_K) \big)$
is the $p$-adic rank of $E_K$ (i.e., the $\Z_p$-rank of the closure of
the image $\iota_P^{}(E_K)$ of $E_K$ in the group of local principal units $U_{K,P}$, 
where $\iota_P^{}$ is the diagonal embedding; see \S\;\ref{section2}).

\smallskip
The \v Safarevi\v c and reflection formulas, generalized with decomposition, 
may be obtained via \cite[Exercise II.5.4.1]{Gr3}\,[Gr2003] or other classical references.

\smallskip
In general, $\wt r_{K,S}^{}$ is non-obvious and varies from $0$ to $r_2+1$ 
(see Wingberg \cite{Win1,Win2}\,[Win1989-1991], Yamagishi \cite{Ya}\,[Yam1993], 
Maire \cite{M01,M02,M1}\,[Mai2002-2003-2005], Labute \cite{Lab}\,[Lab2006], \cite{LM}\,[LM2011],
Vogel \cite{Vog}\,[Vog2007] for some results and cases where 
${\mathcal G}_{K,S}$ may be free with less than $r_2+1$ generators and our
forthcoming numerical results showing that many $\Z_p$-ranks can occur).

\smallskip
For $S=P$ we obtain $\wt r_{K,P}^{} = r_2+1$, under the Leopoldt conjecture, giving
(since $\sum_{{\mathfrak p} \in P} [K_{\mathfrak p}  : \Q_p]=r_1+2\,r_2$):
\begin{equation}\label{rankP}
{\rm rk}_p({\mathcal T}_{K,P}) = {\rm rk}_p(V_{K,P}/K_P^{\times p}) +
\sm_{{\mathfrak p} \in P} \delta_{\mathfrak p} - \delta_K .
\end{equation}

If $S = \es$ then ${\mathcal A}_{K,S} =  {\mathcal T}_{K,S} =: \Cl_K$, the 
$p$-class group of $K$ (ordinary sense).

\begin{remark}
We shall not consider $S$-ramification with $S=P \,\cup\, T$, when $T$ is 
a finite set of tame places, because of the following exact sequence,
{\it under the Leopoldt conjecture} (Neumann \cite{Neum0}\,[Neu1975],
Nguyen Quang Do
\cite[Corollary 4.3]{Ng2}\,[Nqd1986], \cite[Theorem III.4.1.5]{Gr3}\,[Gr2003]),
where the $F_{\mathfrak l}$ are the residue fields:\par

\medskip
\centerline{$1 \too \plus_{{\mathfrak l} \in T} (F_{\mathfrak l}^\times \otimes \Z_p)
\tooo\! {\mathcal T}_{K,P \,\cup\, T} \tooo {\mathcal T}_{K,P} \too 1$.}
\end{remark}

For some specialized applications (about number fields, elliptic 
curves, representation theory, Galois cohomology, Iwasawa's theory, 
$p$-adic $L$-functions) and some recent conjectures, one needs to study 
and compute the above $S$-invariants when $S$ is a subset of $P$ 
and $K/\Q$ not necessarily Galois. Even if $K/\Q$ is Galois, 
the Galois group does not necessarily operate on $S$.
So the classical algebraic considerations
(cohomology, Iwasawa's theory) largely collapse.

\smallskip
So the most tricky invariants in ``incomplete $P$-ramification'' are 
$$\hbox{${\mathcal T}_{K,S}\ \ $ and $\ \ \wt r_{K,S}^{} = 
{\rm rk}_p({\mathcal A}_{K,S}) - {\rm rk}_p({\mathcal T}_{K,S})
= \sm_{{\mathfrak p} \in S} [K_{\mathfrak p}  : \Q_p]
- {\rm dim}_{\Q_p} \big( \Q_p {\rm log}_{S}(E_K) \big)$.}$$

Of course, they highly depend on the decomposition of the 
prime $p$ in the Galois closure of $K$ and probably of specific
$p$-adic properties of units; but it remains the class field theory 
framework above the base field $K$.

\section{General $p$-adic context of $S$-ramification}
Consider a number field $K$ and a given prime $p \geq 2$.
Let $S$ be a subset of the set $P$ of the $p$-places of $K$
and let $H_{K,S}$ be the maximal {\it abelian} $S$-ramified
pro-$p$-extension of $K$; this field contains a (maximal) compositum 
$\wt {K\,}^S$ of $\Z_p$-extensions of $K$ and always the $p$-Hilbert class 
field $H_K:= H_{K,\es}$ of $K$. 

\smallskip
These definitions are given in the ordinary sense when $p=2$ 
(so that the real infinite places of $K$ are not complexified in the 
class fields considered; in other words they are totally split).

\subsection{Fundamental exact sequences}\label{section2}
\smallskip
Let $U_{K,S} := \bigoplus_{{\mathfrak p} \in S} U_{\mathfrak p}$, be the
product of the groups of principal local units of $K_{\mathfrak p}$, 
${\mathfrak p} \in S$, and let $\ov {E\,}^S_{\!\!K}$ be the 
closure of the image $\iota_S^{}(E_K)$ of $E_K$ in $U_{K,S}$.

\smallskip
We denote by $W_{K,S} = \bigoplus_{{\mathfrak p} \in S}\mu^{}_{K_{\mathfrak p}}$ 
the torsion group of the $\Z_p$-module $U_{K,S}$.

\smallskip
If $K/\Q$ is Galois and $S \subsetneq P$, $U_{K,S}$ is not necessarily a Galois module,
which increases the difficulties.

\smallskip
The following $p$-adic result is valid without any assumption 
on $K$ and $S \subseteq P$:

\begin{lemma} \label{exact}
We have the exact sequence:
$$1\to W_{K,S} \big / {\rm tor}_{\Z_p}^{}(\ov {E\,}^S_{\!\!K})  \tooo 
 {\rm tor}_{\Z_p}^{} \big(U_{K,S} \big /\ov {E\,}^S_{\!\!K} \big) 
 \mathop {\tooo}^{{\rm log}_S}  {\rm tor}_{\Z_p}^{}\big({\rm log}_S\big 
(U_{K,S} \big) \big / {\rm log}_S (\ov {E\,}^S_{\!\!K}) \big) \to 0. $$
\end{lemma}

\begin{proof} 
 The surjectivity comes from the fact that if
$u \in U_{K,S}$ is such that $p^n {\rm log}_S(u) = {\rm log}_S(\ov\varepsilon)$, 
$\ov\varepsilon \in \ov {E\,}^S_{\!\!K}$, then $u^{p^n} = \ov\varepsilon \cdot \xi$ 
for $\xi \in W_{K,S}$, hence there exists $m\geq n$ such that $u^{p^m} \in \ov {E\,}^S_{\!\!K}$, 
whence $u$ gives a preimage in ${\rm tor}_{\Z_p}^{} \big(U_{K,S} \big / \ov {E\,}^S_{\!\!K} \big)$.
If $u \in U_{K,S}$ is such that ${\rm log}_S(u) \in {\rm log}_S(\ov {E\,}^S_{\!\!K})$, then
$u = \ov \varepsilon \cdot \xi$ as above, giving the kernel equal to
$\ov {E\,}^S_{\!\!K} \cdot W_{K,S} /\ov {E\,}^S_{\!\!K} = 
W_{K,S}/ {\rm tor}_{\Z_p}^{}(\ov {E\,}^S_{\!\!K})$.
\end{proof}

Put ${\mathcal W}_{K,S}:= W_{K,S} / {\rm tor}_{\Z_p}^{}(\ov {E\,}^S_{\!\!K})$ and
${\mathcal R}_{K,S} := 
{\rm tor}^{}_{\Z_p} \big ({\rm log}_S (U_{K,S}) / {\rm log}_S (\ov {E\,}^S_{\!\!K})\big)$.
Then the exact sequence of Lemma \ref{exact} becomes:
\begin{equation}
1\too {\mathcal W}_{K,S} \tooo  {\rm tor}_{\Z_p}^{} 
\big(U_{K,S} \big / \ov {E\,}^S_{\!\!K} \big) 
\mathop {\tooo}^{{\rm log}_S} {\mathcal R}_{K,S} \too 0.
\end{equation}

\begin{lemma}\label{leo} Let $\mu_K$ be 
the group of roots of unity of $p$-power order of $K$. 
Under the Leopoldt conjecture for $p$ in $K$ we have
${\rm tor}_{\Z_p}^{}(\ov {E\,}^P_{\!\!K}) = \iota_P^{}(\mu_K)$; thus, in that case,
${\mathcal W}_{K,P}= W_{K,P} /\iota_P^{}(\mu_K)$.
\end{lemma}

\begin{proof}
From Jaulent \cite[D\'efinition 2.11, Proposition 2.12]{Ja2}\,[Jau1998] or 
\cite[Theorem III.3.6.2\,(vi)]{Gr3}\,[Gr2003].
\end{proof}

Note that for $S \subsetneq P$, we do not know if ${\rm tor}_{\Z_p}^{}(\ov {E\,}^S_{\!\!K})$
may be larger than $\iota_S^{}(\mu_K)$ (as subgroups of $W_{K,S}$), even under the Leopoldt 
conjecture.

\subsection{Diagram of $S$-ramification}
\smallskip
Consider the following diagram under the Leopoldt 
conjecture for $p$ in $K$.
By definition, ${\mathcal T}_{K,S} =  {\rm tor}_{\Z_p} 
\big ({\mathcal A}_{K,S} \big)$ is the Galois group 
${\rm Gal}(H_{K,S} / \wt {K\,}^S)$; let
$\wt {\Cl_K\!}^{\!S}$ be the subgroup of $\Cl_K$ corresponding to
${\rm Gal}(H_K/ \wt {K\,}^S \!\! \cap H_K)$ by class field theory.
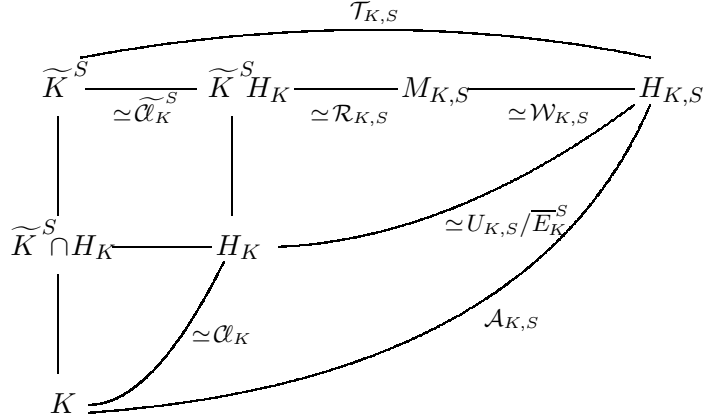
\unitlength=1.05cm 
$$\vbox{\hbox{\hspace{-3.8cm} 
 \begin{picture}(11.5,5.6)
\put(6.5,4.50){\line(1,0){1.3}}
\put(8.7,4.50){\line(1,0){2.1}}
\put(3.85,4.50){\line(1,0){1.4}}
\put(9.2,4.11){\footnotesize$\simeq\! {\mathcal W}_{K,S}$}
\put(4.2,2.50){\line(1,0){1.25}}
\bezier{350}(3.8,4.9)(7.6,5.6)(11.0,4.9)
\put(7.2,5.4){\footnotesize${\mathcal T}_{K,S}$}
\put(3.50,2.9){\line(0,1){1.25}}
\put(3.50,0.9){\line(0,1){1.25}}
\put(5.7,2.9){\line(0,1){1.25}}
\bezier{300}(3.9,0.5)(4.7,0.5)(5.6,2.3)
\put(5.2,1.3){\footnotesize$\simeq \! \Cl_K$}
\bezier{300}(6.3,2.5)(8.5,2.6)(10.8,4.3)
\put(8.4,2.7){\footnotesize$\simeq \! U_{K,S}/\ov {E\,}^S_{\!\!K}$}
\put(10.85,4.4){$H_{K,S}$}
\put(5.4,4.4){$\wt {K\,}^S\! \!H_K$}
\put(7.85,4.4){$M_{K,S}$}
\put(6.7,4.10){\footnotesize$\simeq\! {\mathcal R}_{K,S}$}
\put(4.2,4.10){\footnotesize$\simeq\! \wt {\Cl_K\!}^{\!S}$}
\put(3.3,4.4){$\wt {K\,}^S$}
\put(5.5,2.4){$H_K$}
\put(2.9,2.4){$\wt {K\,}^S\!\! \!\cap \! H_K$}
\put(3.4,0.40){$K$}
\put(8.9,1.5){\footnotesize ${\mathcal A}_{K,S}$}
\bezier{500}(3.9,0.4)(9.5,0.8)(11.0,4.3)
\end{picture}   }} $$
\unitlength=1.0cm
Then from the schema we get:
\begin{equation}\label{partial}
\order {\mathcal T}_{K,S}  = \big [H_K\! :\! \wt {K\,}^S \!\! \cap H_K \big]
 \cdot \order {\rm tor}_{\Z_p} \big( U_{K,S}\big / \ov {E\,}^S_{\!\!K} \big) 
 = \order \wt {\Cl_K\!}^{\!S} 
\cdot \order {\mathcal R}_{K,S} \cdot \order {\mathcal W}_{K,S} .
\end{equation}
Of course, for $p \geq p_0$ (explicit), $\order {\mathcal W}_{K,S} = 
\wt {\Cl_K\!}^{\!S} = 1$, whence ${\mathcal T}_{K,S}={\mathcal R}_{K,S}$.

\begin{remark}\label{rema1}
When $S=P$, we have ${\rm Gal}(H_{K,P} / H_K) 
\simeq U_{K,P}/\ov {E\,}^P_{\!\!K}$,
in which the image of ${\mathcal W}_{K,P}$ fixes $M_{K,P}=: H_K^{\rm bp}$,
the Bertrandias--Payan field, ${\rm Gal}(H_K^{\rm bp} / \wt {K\,}^P)$ being 
the Bertrandias--Payan module as named by Nguyen Quang Do
from \cite{BP}\,[BP1972] on the $p$-cyclic embedding problem. 
Then ${\mathcal R}_{K,P} \simeq {\rm Gal}(H_K^{\rm bp} / \wt {K\,}^P \!\! H_K)$.
This ``normalized regulator'' ${\mathcal R}_{K,P}$ (as a $p$-group or as a $p$-power) 
is closely related to the classical $p$-adic regulator of $K$ 
(see \cite[Proposition 5.2]{Gr7}\,[Gr2018a]).
\end{remark}

\subsection{Local computations}
\smallskip
Recall the following local computation:
\begin{theorem} \cite[Theorem I.4.5 \& Corollary I.4.5.4,
ordinary sense]{Gr3}\,[Gr2003]. \label{thmfond}
For ${\mathfrak p} \mid p$ in $K$ and $j\geq 1$, let $U_{\mathfrak p}^j$
be the group of local units $1+ \ov {\mathfrak p}^j$, where 
$\ov{\mathfrak p}$ is the maximal ideal of the ring of integers of $K_{\mathfrak p}$. 
For $S \subseteq P$, denote by ${\mathfrak m}(S)$ the modulus
$\prod_{{\mathfrak p} \in S}{\mathfrak p}^{e_{\mathfrak p}}$,
where $e_{\mathfrak p}$ is the ramification index of ${\mathfrak p}$ in $K/\Q$.

\smallskip
For a modulus of the form ${\mathfrak m}(S)^n$, $n\geq 0$, let
$\Cl_K({\mathfrak m}(S)^n)$ be the corresponding ray class group
(ordinary sense). 
Then for $m \geq n \geq 0$, we have:
$$0 \leq {\rm rk}_p(\Cl_K({\mathfrak m}(S)^m)) - {\rm rk}_p(\Cl_K({\mathfrak m}(S)^n))
\!\leq\! \sm_{{\mathfrak p} \in S} {\rm rk}_p 
\big((U_{\mathfrak p}^1)^p\,U_{\mathfrak p}^{n \cdot e_{\mathfrak p}} \!/
(U_{\mathfrak p}^1)^p\,U_{\mathfrak p}^{m \cdot e_{\mathfrak p}} \big). $$
\end{theorem}

\begin{corollary} \label{corofond} \cite[Theorem 2.1 \& Corollary 2.2]{Gr8}\,[Gr2017c]
We have:
$${\rm rk}_p(\Cl_K({\mathfrak m}(S)^m)) = 
{\rm rk}_p(\Cl_K({\mathfrak m}(S)^n)) = {\rm rk}_p({\mathcal A}_{K,S}), \ 
\hbox{for all $m \geq n \geq n_0$,} $$

\noindent
where $n_0=3$ for $p=2$ and $n_0=2$ for $p >2$. 
Thus ${\mathcal T}_{K,S}=1$ if and only if 
${\rm rk}_p(\Cl_K({\mathfrak m}(S)^{n_0}))=\wt r_{K,S}^{}$ ($\Z_p$-rank 
of ${\mathcal A}_{K,S}$).
\end{corollary}

\begin{proof} It is sufficient to get, for some fixed $n \geq 0$:
$$\hbox{$(U_{\mathfrak p}^1)^p\,U_{\mathfrak p}^{n \cdot e_{\mathfrak p}} 
= (U_{\mathfrak p}^1)^p$, \  for all ${\mathfrak p} \in S$, }$$

\noindent
hence $U_{\mathfrak p}^{n \cdot e_{\mathfrak p}} \subseteq 
(U_{\mathfrak p}^1)^p$ for all ${\mathfrak p} \in S$; 
indeed, we then have:
$$\hbox{${\rm rk}_p(\Cl_K({\mathfrak m}(S)^n)) = {\rm rk}_p(\Cl_K({\mathfrak m}(S)^m)) = 
\wt r_{K,S}^{} + {\rm rk}_p({\mathcal T}_{K,S})$ as $m\to\infty$, } $$

\noindent
giving ${\rm rk}_p(\Cl_K({\mathfrak m}(S)^n)) = \wt r_{K,S}^{}+ {\rm rk}_p({\mathcal T}_{K,S})$ 
for such $n$. 

\smallskip
The condition $U_{\mathfrak p}^{n \cdot e_{\mathfrak p}} \subseteq 
(U_{\mathfrak p}^1)^p$ is fulfilled as soon as
$n \cdot e_{\mathfrak p} > \Frac{p\cdot e_{\mathfrak p}}{p-1}$, whence
$n > \Frac{p}{p-1}$ (Fesenko--Vostokov
\cite[Chapter I, \S\,5.8, Corollary 2]{FV}\,[FV2002])
giving the value of $n_0$; furthermore, $\Cl_K({\mathfrak m}(S)^{n_0})$ gives the 
$p$-rank of ${\mathcal T}_{K,S}$ as soon as the $\Z_p$-rank $\wt r_{K,S}^{}$ is known.
\end{proof}

\subsection{Practical computation of $\wt r_{K,S}^{}$}\label{comprtild}
\smallskip
Let $S \subseteq P$. From \eqref{rank}, we have:
$\wt r_{K,S}^{} =\sm_{{\mathfrak p} \in S} 
[K_{\mathfrak p} : \Q_p] - r_{K,S}^{}$,
where $r_{K,S}^{} := {\rm dim}_{\Q_p} \big( \Q_p {\rm log}_{S}(E_K) \big)$.

\quad (i) In \cite{M01,M02}\,[Mai2002-2003] Maire has given, in the relative Galois case, 
some results about $r_{K,S}^{}$ depending on Schanuel's conjecture 
and the use of the representation $\Q_p {\rm log}_{S}(E_K)$ from the 
results of Jaulent \cite{Ja00}\,[Jau1985].

\smallskip
\quad (ii) In the Galois case, this rank has been studied by Nelson \cite{Nel}\,[Nel2013]
giving formulas (or  lower bounds) under the $p$-adic Schanuel 
conjecture.

\smallskip
\quad (iii) We have proposed, in \cite[III,\,\S\,4\,(f)]{Gr3}\,[Gr2003], a conjecture
and a calculation process in the general non-Galois case using a 
Galois descent from the
Galois closure $N$ of $K$ and the family of decomposition groups
of the places of $N$ above $p$ and $\infty$. If $K/\Q$ is Galois then
(with $\Sigma := P \setminus S$):

\smallskip
\centerline{${\rm rk}_{\Z_p} \big({\rm Gal}(\wt {K\,}^P\!\!/\wt {K\,}^S) \big) =
\sm_{{\mathfrak p} \in \Sigma} [K_{\mathfrak p} : \Q_p] - 
{\rm dim}_{\Q_p}\big(\Q_p {\rm log}_{P}({\mathcal E}_{K,S}) \big)$,}

\noindent
where ${\mathcal E}_{K,S} := \big\{ \varepsilon \in E_K \otimes \Z_p,\ 
\iota_{\mathfrak p}^{}( \varepsilon) = 1,\ \forall {\mathfrak p} \in S \big\}$
and $\iota_{\mathfrak p}^{} : E_K \otimes \Z_p \to U_{\mathfrak p}^1$.

\smallskip
But all these similar approaches
are difficult for programming and not so obvious for random $K$ and $S$
because of conjectural aspects; so we shall preferably give extensive 
computations via PARI/GP \cite{P} since ray class fields are well computed. 
But it remains the problem of justification of the ``computing'' of $\wt r_{K,S}^{}$, 
when no theoretical value is known (see another explicit numerical method in 
\cite[\S\,III.5, Theorem 5.2]{Gr3}\,[Gr2003]).

\smallskip
We conclude by the following comments:

\begin{remark} If ${\mathcal T}_{K,P}=1$ (i.e., the field $K$ is called 
$p$-rational as proposed by Movahhedi in \cite{Mo1,Mo2}\,[Mov1988-1990]),
this does not imply ${\mathcal T}_{K,S}=1$ for $S \subsetneq P$
(the numerical examples will show many cases). 
In the opposite situation, we may have ${\mathcal T}_{K,P} \ne1$,
but often ${\mathcal T}_{K,S}=1$ for $S \subsetneq P$.

\smallskip
This intricate aspects have been studied by Maire \cite[Section 3]{M1}\,[Mai2005]
in which he introduces the ``$S$-cohomologcal condition''
${\rm H}^2({\mathcal G}_{K,S}, \Q_p/\Z_p)=0$ (knowing that 
${\mathcal G}_{K,S}$ is a free pro$p$-group if and only if 
${\rm H}^2({\mathcal G}_{K,S}, \Q_p/\Z_p)$ and ${\mathcal T}_{K,S}$
are trivial) and that of ``$S$-arithmetical condition'' 
($E_K \otimes \Z_p \to U_{K,S}$ injective), and compare them,
which of course coincide for $S=P$; we know that the 
$S$-arithmetical condition implies the $S$-cohomologcal one.

\smallskip
We shall speak of $S$-rationality, when ${\mathcal T}_{K,S}=1$ 
for $S \subseteq P$, even if this may be rather ambiguous when
$S \subsetneq P$ because of the above observations; one must 
understand this as a ``free $S$-ramification'' over $K$ (i.e., giving 
a free abelian $S$-ramified pro-$p$-extension $H_{K,S}/K$).
This is also justified by the fact that many variants of the definition have been 
given, as those of Jaulent--Sauzet \cite{JS1,JS2}\,[JS1997-2000], 
Bourbon--Jaulent \cite{BJ}\,[BJ2013], 
where are defined and studied the case of singleton $S=\{{\mathfrak p}\}$ or that 
of the ``$2$-birationality'' of quadratic extensions of totally real fields
when $S=\{{\mathfrak p}, {\mathfrak p}'\}$.
\end{remark}

\section{Algorithmic approach of $S$-ramification}

The principle is to consider a modulus ${\mathfrak m}_S := 
\prod_{{\mathfrak p} \in S}{\mathfrak p}^{\lambda_{\mathfrak p}}$,
$S \subseteq P$, with $\lambda_{\mathfrak p} \gg 0$ for all 
${\mathfrak p} \in S$ to ``read'' the structure of ${\mathcal A}_{K,S}$
on the ray class group $\Cl_K({\mathfrak m}_S)$. 
The practice shows that the more convenient modulus is of the form:

\smallskip
\centerline{$\Big(\prd_{{\mathfrak p} \in S}{\mathfrak p}^{e_{\mathfrak p}}\Big)^n$,}

\noindent
where $e_{\mathfrak p}$ is the ramification index of ${\mathfrak p}$ in $K/\Q$
and $n \gg 0$. Of course, this modulus is $(p^n)$ only for $S=P$;
so we must use the ideal decomposition of $p$ in $K$, given by PARI/GP,
and compute everywhere with ideals.

\subsection{Main program computing ${\mathcal T}_{K,S}$ and $\wt r_{K,S}^{}$}
\smallskip
\subsubsection{The PARI/GP program} \label{PP}
\smallskip
\footnotesize
\begin{verbatim}
==========================================================================================
{P=x^3+197*x^2+718*x+508;if(polisirreducible(P)==0,break);print(P);bp=2;Bp=5000;n0=6;
K=bnfinit(P,1);forprime(p=bp,Bp,n=n0+floor(30/p);print();print("p=",p);F=idealfactor(K,p);
d=component(matsize(F),1);F1=component(F,1);for(j=1,d,print(component(F1,j)));
for(z=2^d,2^(d+1)-1,bin=binary(z);mod=List;for(j=1,d,listput(mod,component(bin,j+1),j));
M=1;for(j=1,d,ch=component(mod,j);if(ch==1,F1j=component(F1,j);ej=component(F1j,3);
F1j=idealpow(K,F1j,ej);M=idealmul(K,M,F1j)));Idn=idealpow(K,M,n);Kpn=bnrinit(K,Idn);
Hpn=component(component(Kpn,5),2);L=List;e=component(matsize(Hpn),2);R=0;
for(k=1,e,c=component(Hpn,e-k+1);w=valuation(c,p);if(w>0,R=R+1;listinsert(L,p^w,1)));
print("S=",mod," rk(A_S)=",R," A_S=",L)))}
==========================================================================================
\end{verbatim}
\normalsize

\subsubsection{Instructions for use and illustrations}
\smallskip
See the Note at the end of Section \ref{note}.
The reader has only to copy and past the verbatim of the program and to
use a ``terminal session via Sage'', on his or her computer, or a cell in the page $\ $ 
{\url{http://pari.math.u-bordeaux.fr/gp.html}}

The programs in this article can be directly copied and pasted at:

\url{https://www.dropbox.com/s/1srmksbr2ujf40i/Incomplete%20p-ramification.pdf?dl=0}

\smallskip
It is assumed that the irreducible monic polynomial ${\sf P}$ 
defining $K$ is given and that the interval ${\sf [bp, Bp]}$
of tested primes $p$ is also given by the user.

\smallskip
\quad (i) The program computes the decomposition of $p$ into ${\sf d}$ prime ideals; 
for instance, the following data gives, for ${\sf P=x^3 + 197*x^2 + 718*x + 508}$
and $p=2$, the decomposition $(p) = {\mathfrak p} {\mathfrak p}'$ in $\Q(x)$, 
using ${\sf idealfactor(K,p)}$:

\footnotesize
\begin{verbatim}
[2, [-65, 0, 1]~, 1, 1, [0, 0, -1]~] 
[2, [0, 0, 1]~, 1, 2, [0, 1, 0]~]
\end{verbatim}
\normalsize

Recall that for an ideal as ${\sf [2, [0, 0, 1]\ \wt{}\ , 1, 2, [0, 1, 0]\ \wt{}\ ]}$,
the $3$th component is its ramification index, the $4$th component
is its residue degree. For the computation of the modulus
${\mathfrak m}_S$ (to be considered at the power $n$), we replace 
each prime ideal ${\mathfrak p} \in S$ by ${\mathfrak p}^{e_{\mathfrak p}}$
using the function ${\sf idealpow}$.

\smallskip
\quad (ii) For each modulus ${\mathfrak m}_S=\prod_{{\mathfrak p} \in S}
{\mathfrak p}^{e_{\mathfrak p} \cdot n}$, the program gives 
${\rm rk}_p({\mathcal A}_{K,S})$ and the $\Z$-structure of 
${\mathcal A}_{K,S}/{\mathcal A}_{K,S}^{p^N}$, for $N$ of the 
order of $n$, under the form:
$${\mathcal A}_{K,S} = [a_1,\ldots,a_r ; \ b_1, \ldots, b_t], $$

\noindent
where the coefficients $a_1, \ldots,a_r$
increase (resp. the coefficients $b_1, \ldots, b_t$ stabilize) as 
the exponent $n$ increses, so in the 
non-ambiguous cases, $b_1, \ldots, b_t$ give the group-invariants of 
${\mathcal T}_{K,S}$ and $r$ is the $p$-rank $\wt r_{K,S}^{}$ of 
${\rm Gal}(\wt {K\,}^S\!/K)$. 

\smallskip
Of course, if the rank $\wt r_{K,S}^{}$
is not certain, we can not, in a mathematical point of view,
deduce the structure of ${\mathcal T}_{K,S}$; but in practice
the information is correct since one can always verify, with the program, 
the stabilization of the invariants $b_j$ whereas the $a_i$ increase linearly
to infinity.

\smallskip
\quad (iii) The symbolic data $S=[\delta_1, \ldots , \delta_d]$, $\delta_i \in \{0,1\}$,
indicates that the $S$-modulus considered is:
$${\mathfrak m}_S =\Big( \prd_{i=1}^d 
{\mathfrak p}_i^{e_{{\mathfrak p}_i}^{} \!\cdot \,\delta_i^{} } \Big)^n. $$

We have choosen ${\sf n= floor \big(n_0+\frac{30}{p} \big)}$ to get small values
when $p \gg 0$ but larger ones for small $p$ (especially $p=2$ 
giving possibly huge $\order {\mathcal T}_{K,S}$). The parameter
$n_0$ may be increased at will (here $n_0=6$).

\smallskip
There are $2^{\order S}$ distinct sets $S$ parametrized with the binary writing 
of the integers $z \in [0, 2^d-1]$.

\smallskip
For $S=[0, \ldots , 0]$ one obtains the structure of the $p$-class group $\Cl_K$.

\smallskip
\quad (iv) We illustrate the program with an example where $K$ (a totally 
real cubic field) is not $S$-rational for some small $p$ and some 
$S \subseteq P$; but in almost all cases, $K$ is $S$-rational.

\begin{remark}
We do not compute the Galois group associated to the given
polynomial, nor the discriminant or the fundamental units; otherwise, 
the reader has only to add if necessary the instructions:
\footnotesize
\begin{verbatim}
print("Galois :",polgalois(P));
print("Discriminant: ",factor(component (component(K,7), 3)));
print("Fundamental system of units: ",component(component(K,8),5));
\end{verbatim}

\normalsize
\noindent
giving, for the Galois group and the discriminant:

\smallskip
${\sf Galois group=[6, -1, 1,}$ ${\sf "S3"]}$ in the PARI/GP 
notation\,\footnote{See: \url{http://galoisdb.math.upb.de/home}} and 
${\sf Discriminant = [769, 1; 1390573, 1])}$.
\end{remark}

\footnotesize
\begin{verbatim}
P=x^3 + 197*x^2 + 718*x + 508
p=2 
[2, [-65, 0, 1]~, 1, 1, [0, 0, -1]~]
[2, [0, 0, 1]~, 1, 2, [0, 1, 0]~]
S=[0, 0] rk(A_S)=0 A_S=[]
S=[0, 1] rk(A_S)=1 A_S=[4]
S=[1, 0] rk(A_S)=0 A_S=[]
S=[1, 1] rk(A_S)=3 A_S=[274877906944, 4, 2]
p=3 
[3, [3, 0, 0]~, 1, 3, 1]
S=[0] rk(A_S)=0 A_S=[]
S=[1] rk(A_S)=2 A_S=[22876792454961, 3]
p=5 
[5, [-68, 0, 1]~, 1, 1, [-1, 2, -1]~]
[5, [12589, 2, -196]~, 1, 2, [2, 0, 1]~]
S=[0, 0] rk(A_S)=0 A_S=[]
S=[0, 1] rk(A_S)=1 A_S=[390625]
S=[1, 0] rk(A_S)=0 A_S=[]
S=[1, 1] rk(A_S)=2 A_S=[19073486328125, 390625]
p=7 
[7, [-65, 0, 1]~, 1, 1, [3, 2, 1]~]
[7, [12519, 2, -195]~, 1, 2, [-2, 0, 1]~]
S=[0, 0] rk(A_S)=0 A_S=[]
S=[0, 1] rk(A_S)=1 A_S=[7]
S=[1, 0] rk(A_S)=0 A_S=[]
S=[1, 1] rk(A_S)=2 A_S=[33232930569601, 7]
p=11 
[11, [11, 0, 0]~, 1, 3, 1]
S=[0] rk(A_S)=0 A_S=[]
S=[1] rk(A_S)=2 A_S=[3138428376721, 11]
p=13 
[13, [13, 0, 0]~, 1, 3, 1]
S=[0] rk(A_S)=0 A_S=[]
S=[1] rk(A_S)=1 A_S=[1792160394037]
(...)
p=127 
[127, [-66, 0, 1]~, 1, 1, [-16, 2, 2]~]
[127, [16240, 2, -252]~, 1, 2, [61, 0, 1]~]
S=[0, 0] rk(A_S)=0 A_S=[]
S=[0, 1] rk(A_S)=1 A_S=[127]
S=[1, 0] rk(A_S)=0 A_S=[]
S=[1, 1] rk(A_S)=2 A_S=[532875860165503, 127]
p=1571 
[1571, [275, 0, 1]~, 1, 1, [-418, 2, -339]~]
[1571, [21576, 2, -339]~, 1, 2, [275, 0, 1]~]
S=[0, 0] rk(A_S)=0 A_S=[]
S=[0, 1] rk(A_S)=1 A_S=[1571]
S=[1, 0] rk(A_S)=0 A_S=[]
S=[1, 1] rk(A_S)=2 A_S=[23617465807865561078891, 1571]
p=1759 
[1759, [1759, 0, 0]~, 1, 3, 1]
S=[0, 0] rk(A_S)=0 A_S=[]
S=[1] rk(A_S)=2 A_S=[52102777604679963122719, 1759]
p=3371 
[3371, [-295, 0, 1]~, 1, 1, [-1597, 2, 231]~]
[3371, [-121, 0, 1]~, 1, 1, [355, 2, 57]~]
[3371, [415, 0, 1]~, 1, 1, [38, 2, -479]~]
S=[0, 0, 0] rk(A_S)=0 A_S=[]
S=[0, 0, 1] rk(A_S)=0 A_S=[]
S=[0, 1, 0] rk(A_S)=0 A_S=[]
S=[0, 1, 1] rk(A_S)=1 A_S=[3371]
S=[1, 0, 0] rk(A_S)=0 A_S=[]
S=[1, 0, 1] rk(A_S)=1 A_S=[3371]
S=[1, 1, 0] rk(A_S)=1 A_S=[3371]
S=[1, 1, 1] rk(A_S)=2 A_S=[4946650964538063853923491, 3371]
\end{verbatim}
\normalsize

\smallskip
If, for the remarquable case $p=5$, one has some doubt, one increases $n$,
which gives (for $n=50$):

\smallskip
\footnotesize
\begin{verbatim} 
[5, [-68, 0, 1]~, 1, 1, [-1, 2, -1]~]
[5, [12589, 2, -196]~, 1, 2, [2, 0, 1]~]
S=[0, 0] rk(A_S)=0 A_S=[]
S=[0, 1] rk(A_S)=1 A_S=[390625]
S=[1, 0] rk(A_S)=0 A_S=[]
S=[1, 1] rk(A_S)=2 A_S=[17763568394002504646778106689453125, 390625]
\end{verbatim}
\normalsize

\smallskip
Whence ${\mathcal T}_{K,S} \simeq \Z/ 5^8\Z$ for $S_1=\{{\mathfrak p}\}$
(for the prime of residue degree $2$) and $S_2=P$. Note that once the substantial
computation of ${\sf K=bnfinit(P,1)}$ (giving all the basic information about the field)
is done, very large values of $n$ do not increase much the execution time; so any
skeptical user can make $n \to \infty$ to see that only the data $390625$ remains
constant.

\smallskip
(v) In \cite[\S\,9.1]{Gr10}\,[Gr2019a] we have used some special families of polynomials
(e.g., Lecacheux--Washington ones) in which we can force the $p$-adic 
regulator to be $p$-adicaly close to $0$ at will; but we must take the parameter $n$
in proportion, even if here the $\Z_p$-ranks of the ${\mathcal A}_{K,S}$ are obvious,
since $K$ is totally real, giving finite groups except for $S=P$ where 
${\rm rk}_{\Z_p}({\mathcal A}_{K,P})=1$:

\smallskip
\footnotesize
\begin{verbatim}
P=x^3-134480895*x^2-263169*x-1
p=2
[2, [0, 0, 1]~, 1, 1, [1, 0, 1]~]
[2, [0, 1, 0]~, 1, 1, [1, 1, 0]~]
[2, [2, 1, 1]~, 1, 1, [1, 1, 1]~]
S=[0, 0, 0] rk(A_S)=6 A_S=[16, 16, 2, 2, 2, 2]
S=[0, 0, 1] rk(A_S)=6 A_S=[512, 16, 8, 2, 2, 2]
S=[0, 1, 0] rk(A_S)=6 A_S=[512, 16, 8, 2, 2, 2]
S=[0, 1, 1] rk(A_S)=6 A_S=[1024, 512, 8, 8, 2, 2]
S=[1, 0, 0] rk(A_S)=6 A_S=[512, 16, 8, 2, 2, 2]
S=[1, 0, 1] rk(A_S)=6 A_S=[1024, 512, 8, 8, 2, 2]
S=[1, 1, 0] rk(A_S)=6 A_S=[1024, 512, 8, 8, 2, 2]
S=[1, 1, 1] rk(A_S)=7 A_S=[9444732965739290427392, 1024, 1024, 8, 8, 2, 2]

x^3-7625984944841*x^2-387459856*x-1
p=3
[3, [1, -1, -1]~, 1, 1, [0, 1, 1]~]
[3, [2, 1, 0]~, 1, 1, [1, 1, 0]~]
[3, [2541994975055, -19683, 1]~, 1, 1, [-1, 0, -1]~]
S=[0, 0, 0] rk(A_S)=4 A_S=[27, 9, 3, 3]
S=[0, 0, 1] rk(A_S)=4 A_S=[177147, 9, 3, 3]
S=[0, 1, 0] rk(A_S)=4 A_S=[177147, 9, 3, 3]
S=[0, 1, 1] rk(A_S)=4 A_S=[177147, 59049, 3, 3]
S=[1, 0, 0] rk(A_S)=4 A_S=[177147, 9, 3, 3]
S=[1, 0, 1] rk(A_S)=4 A_S=[177147, 59049, 3, 3]
S=[1, 1, 0] rk(A_S)=4 A_S=[177147, 59049, 3, 3]
S=[1, 1, 1] rk(A_S)=5 A_S=[834385168331080533771857328695283, 177147, 59049, 3, 3]
P=x^3-1628427439432947*x^2-13841522500*x-1
p=7
[7, [1, -3, -3]~, 1, 1, [0, 1, 1]~]
[7, [4, 3, 0]~, 1, 1, [1, 1, 0]~]
[7, [542809146438439, -117649, 1]~, 1, 1, [2, 0, 2]~]
S=[0, 0, 0] rk(A_S)=2 A_S=[7, 7]
S=[0, 0, 1] rk(A_S)=2 A_S=[117649, 7]
S=[0, 1, 0] rk(A_S)=2 A_S=[117649, 7]
S=[0, 1, 1] rk(A_S)=3 A_S=[117649, 16807, 7]
S=[1, 0, 0] rk(A_S)=2 A_S=[117649, 7]
S=[1, 0, 1] rk(A_S)=3 A_S=[117649, 16807, 7]
S=[1, 1, 0] rk(A_S)=3 A_S=[117649, 16807, 7]
S=[1, 1, 1] rk(A_S)=4 A_S=[3219905755813179726837607, 117649, 16807, 7]
\end{verbatim}
\normalsize

\subsubsection{Example with $p$ totally split in degree $5$}
\smallskip
For $P=x^5-5$, $n_0=8$, and $p=31$ (totally split) one finds one
case of non $S$--rationality: 

${\sf S=[1, 0, 0, 0, 1]\ \ rk(A_S)=1 \  \ A_S=[961]}$, i.e., 
$\wt r_{K,S}^{}=0$, ${\mathcal T}_{K,S} \simeq \Z/ 31^2\Z$:

\footnotesize
\begin{verbatim}
[31, [-14, 1, 0, 0, 0]~, 1, 1, [7, -15, 10, 14, 1]~]
[31, [-7, 1, 0, 0, 0]~, 1, 1, [14, 2, -13, 7, 1]~]
[31, [3, 1, 0, 0, 0]~, 1, 1, [-12, 4, 9, -3, 1]~]
[31, [6, 1, 0, 0, 0]~, 1, 1, [-6, 1, 5, -6, 1]~]
[31, [12, 1, 0, 0, 0]~, 1, 1, [-3, 8, -11, -12, 1]~]
S=[0, 0, 0, 0, 0] rk(A_S)=0 A_S=[]
S=[0, 0, 0, 0, 1] rk(A_S)=0 A_S=[]
S=[0, 0, 0, 1, 0] rk(A_S)=0 A_S=[]
S=[0, 0, 0, 1, 1] rk(A_S)=0 A_S=[]
S=[0, 0, 1, 0, 0] rk(A_S)=0 A_S=[]
S=[0, 0, 1, 0, 1] rk(A_S)=0 A_S=[]
S=[0, 0, 1, 1, 0] rk(A_S)=0 A_S=[]
S=[0, 0, 1, 1, 1] rk(A_S)=1 A_S=[27512614111]
S=[0, 1, 0, 0, 0] rk(A_S)=0 A_S=[]
S=[0, 1, 0, 0, 1] rk(A_S)=0 A_S=[]
S=[0, 1, 0, 1, 0] rk(A_S)=0 A_S=[]
S=[0, 1, 0, 1, 1] rk(A_S)=1 A_S=[27512614111]
S=[0, 1, 1, 0, 0] rk(A_S)=0 A_S=[]
S=[0, 1, 1, 0, 1] rk(A_S)=1 A_S=[27512614111]
S=[0, 1, 1, 1, 0] rk(A_S)=1 A_S=[27512614111]
S=[0, 1, 1, 1, 1] rk(A_S)=2 A_S=[27512614111, 27512614111]
S=[1, 0, 0, 0, 0] rk(A_S)=0 A_S=[]
S=[1, 0, 0, 0, 1] rk(A_S)=1 A_S=[961]
S=[1, 0, 0, 1, 0] rk(A_S)=0 A_S=[]
S=[1, 0, 0, 1, 1] rk(A_S)=1 A_S=[27512614111]
S=[1, 0, 1, 0, 0] rk(A_S)=0 A_S=[]
S=[1, 0, 1, 0, 1] rk(A_S)=1 A_S=[27512614111]
S=[1, 0, 1, 1, 0] rk(A_S)=1 A_S=[27512614111]
S=[1, 0, 1, 1, 1] rk(A_S)=2 A_S=[27512614111, 27512614111]
S=[1, 1, 0, 0, 0] rk(A_S)=0 A_S=[]
S=[1, 1, 0, 0, 1] rk(A_S)=1 A_S=[27512614111]
S=[1, 1, 0, 1, 0] rk(A_S)=1 A_S=[27512614111]
S=[1, 1, 0, 1, 1] rk(A_S)=2 A_S=[27512614111, 27512614111]
S=[1, 1, 1, 0, 0] rk(A_S)=1 A_S=[27512614111]
S=[1, 1, 1, 0, 1] rk(A_S)=2 A_S=[27512614111, 27512614111]
S=[1, 1, 1, 1, 0] rk(A_S)=2 A_S=[27512614111, 27512614111]
S=[1, 1, 1, 1, 1] rk(A_S)=3 A_S=[27512614111, 27512614111, 27512614111]
\end{verbatim}
\normalsize

\subsubsection{Example with $p$ totally split in degree $7$}
\smallskip
For the polynomial $P=x^7-7$ and $p=43$, one finds two cases:

\footnotesize
\begin{verbatim}
[43, [-18, 1, 0, 0, 0, 0, 0]~, 1, 1, [-2, 19, 13, -16, -20, 18, 1]~]
[43, [-7, 1, 0, 0, 0, 0, 0]~, 1, 1, [1, -6, -7, -1, 6, 7, 1]~]
[43, [9, 1, 0, 0, 0, 0, 0]~, 1, 1, [4, -10, -18, 2, -5, -9, 1]~]
[43, [13, 1, 0, 0, 0, 0, 0]~, 1, 1, [16, 12, 9, -4, -3, -13, 1]~]
[43, [14, 1, 0, 0, 0, 0, 0]~, 1, 1, [21, 20, 17, 8, -19, -14, 1]~]
[43, [15, 1, 0, 0, 0, 0, 0]~, 1, 1, [11, 5, 14, -21, 10, -15, 1]~]
[43, [17, 1, 0, 0, 0, 0, 0]~, 1, 1, [-8, 3, 15, -11, -12, -17, 1]~]
(...)
S=[0, 1, 0, 1, 0, 0, 1] rk(A_S)=1 A_S=[43]
S=[1, 1, 0, 0, 1, 0, 0] rk(A_S)=1 A_S=[43]
\end{verbatim}
\normalsize

\noindent
i.e., $\wt r_{K,S}^{}=0$ and ${\mathcal T}_{K,S} \simeq \Z/ 43\Z$ for 
the two above cases.
For the other modulus, ${\mathcal T}_{K,S}=1$.

\subsubsection{Example with a field discovered by Jaulent--Sauzet}
\smallskip
In \cite{JS1}\,[JS1997], some numerical examples of 
$\{{\mathfrak l}\}(=\{{\mathfrak p}\})$-rational fields, which are not $p$-rational,
are given; of course this corresponds to a suitable choice of $S=\{{\mathfrak p}\}$ 
and we give the case of the field defined by the polynomial:
$$P=x^{10}+19x^8+8x^7+130x^6+16x^5+166x^4-888x^3-15x^2+432x+243$$ 
for $p=3$:

\footnotesize
\begin{verbatim}
[3, [-1, 1, 0, 0, 1, 1, -1, 0, 0, -1]~, 2, 1, [2, 0, 2, 1, 2, 0, 1, 1, 2, 1]~]
[3, [-1, 1, 0, 1, 1, 0, -1, 0, 0, -1]~, 2, 1, [2, 0, 1, 2, 1, 2, 1, 1, 2, 1]~]
[3, [-5, 14, -4, -2, 5, 5, 13, -13, 2, 6]~, 2, 3, [0, 1, 1, 1, -1, -1, -1, -1, -1, 1]~]

S=[0, 0, 0] rk(A_S)=0 A_S=[]
S=[0, 0, 1] rk(A_S)=2 A_S=[14348907,14348907]
S=[0, 1, 0] rk(A_S)=0 A_S=[]
S=[0, 1, 1] rk(A_S)=5 A_S=[14348907,14348907,14348907,14348907, 3]
S=[1, 0, 0] rk(A_S)=0 A_S=[]
S=[1, 0, 1] rk(A_S)=5 A_S=[14348907,14348907,14348907,14348907, 3]
S=[1, 1, 0] rk(A_S)=1 A_S=[27]
S=[1, 1, 1] rk(A_S)=8 A_S=[14348907,14348907,14348907,14348907,14348907,14348907, 3, 3]
\end{verbatim}
\normalsize

\noindent
which is indeed $\{{\mathfrak p}\}$-rational for each prime ideal ${\mathfrak p}$, but 
the field is not $3$-rational since ${\mathcal T}_{K,P}\simeq \Z/3\Z \times  \Z/3\Z$.

\smallskip
Note the case ${\mathcal A}_{K,S} = {\mathcal T}_{K,S} \simeq \Z/27 \Z$.

\smallskip
Many other numerical examples are available in \cite[\S\,3.c]{JS1}\,[JS1997].

\subsubsection{Abelian fields with ${\mathcal T}_{K,S}=1$ but ${\mathcal T}_{K,P}\ne 1$}
\smallskip
We consider for this the cyclotomic field $\Q(\mu_{24})$. The following program 
may be used for any abelian field given by ${\sf polcyclo(N)}$
or ${\sf polsubcyclo(N,d)}$ giving the suitable polynomials of degree $d$
dividing $\varphi(N)$:

\smallskip
\footnotesize
\begin{verbatim}
{P=polcyclo(24);bp=2;Bp=500;n0=8;K=bnfinit(P,1);forprime(p=bp,Bp,n=n0+floor(30/p);print();
print("p=",p);F=idealfactor(K,p);d=component(matsize(F),1);F1=component(F,1);for(j=1,d,
print(component(F1,j)));for(z=2^d,2*2^d-1,bin=binary(z);mod=List;for(j=1,d,
listput(mod,component(bin,j+1),j));M=1;for(j=1,d,ch=component(mod,j);if(ch==1,F1j=component(F1,j);
ej=component(F1j,3);FF1j=idealpow(K,F1j,ej);M=idealmul(K,M, FF1j)));Idn=idealpow(K,M,n);
Kpn=bnrinit(K,Idn);Hpn=component(component(Kpn,5),2);L=List;e=component(matsize(Hpn),2);R=0;
for(k=1,e,c=component(Hpn,e-k+1);w=valuation(c,p);if(w>0,R=R+1;listinsert(L,p^w,1)));
print("S=",mod," rk(A_S)=",R," A_S=",L)))}

p=3
[3, [-1, 0, -1, 0, 1, 0, 0, 0]~, 2, 2, [-1, -1, 1, 1, 1, 1, 0, 0]~]
[3, [-1, 0, 1, 0, 1, 0, 0, 0]~, 2, 2, [-1, -1, -1, -1, 1, 1, 0, 0]~]
S=[0, 0] rk(A_S)=0 A_S=[]
S=[0, 1] rk(A_S)=1 A_S=[22876792454961]
S=[1, 0] rk(A_S)=1 A_S=[22876792454961]
S=[1, 1] rk(A_S)=6 A_S=[68630377364883,22876792454961,22876792454961,22876792454961,22876792454961, 3]

p=7
[7, [-3, 0, -1, 0, 1, 0, 0, 0]~, 1, 2, [2, -3, -3, 1, -3, 1, 0, 0]~]
[7, [-3, 0, 1, 0, 1, 0, 0, 0]~, 1, 2, [2, -3, 3, -1, -3, 1, 0, 0]~]
[7, [2, 0, -2, 0, 1, 0, 0, 0]~, 1, 2, [-3, 2, -3, 2, 2, 1, 0, 0]~]
[7, [2, 0, 2, 0, 1, 0, 0, 0]~, 1, 2, [-3, 2, 3, -2, 2, 1, 0, 0]~]
S=[0, 0, 0, 0] rk(A_S)=0 A_S=[]
S=[0, 0, 0, 1] rk(A_S)=0 A_S=[]
S=[0, 0, 1, 0] rk(A_S)=0 A_S=[]
S=[0, 0, 1, 1] rk(A_S)=2 A_S=[4747561509943, 7]
S=[0, 1, 0, 0] rk(A_S)=0 A_S=[]
S=[0, 1, 0, 1] rk(A_S)=2 A_S=[4747561509943,4747561509943]
S=[0, 1, 1, 0] rk(A_S)=2 A_S=[4747561509943, 7]
S=[0, 1, 1, 1] rk(A_S)=4 A_S=[4747561509943,4747561509943,4747561509943, 7]
S=[1, 0, 0, 0] rk(A_S)=0 A_S=[]
S=[1, 0, 0, 1] rk(A_S)=2 A_S=[4747561509943, 7]
S=[1, 0, 1, 0] rk(A_S)=2 A_S=[4747561509943,4747561509943]
S=[1, 0, 1, 1] rk(A_S)=4 A_S=[4747561509943,4747561509943,4747561509943, 7]
S=[1, 1, 0, 0] rk(A_S)=2 A_S=[4747561509943, 7]
S=[1, 1, 0, 1] rk(A_S)=4 A_S=[4747561509943,4747561509943,4747561509943, 7]
S=[1, 1, 1, 0] rk(A_S)=4 A_S=[4747561509943,4747561509943,4747561509943, 7]
S=[1, 1, 1, 1] rk(A_S)=6 A_S=[4747561509943,4747561509943,4747561509943,4747561509943,4747561509943, 7]

p=13
[13, [-6, 0, 0, 0, 1, 0, 0, 0]~, 1, 2, [2, 6, 0, 0, -4, 1, 0, 0]~]
[13, [-2, 0, 0, 0, 1, 0, 0, 0]~, 1, 2, [6, 2, 0, 0, 3, 1, 0, 0]~]
[13, [2, 0, 0, 0, 1, 0, 0, 0]~, 1, 2, [-6, -2, 0, 0, 3, 1, 0, 0]~]
[13, [6, 0, 0, 0, 1, 0, 0, 0]~, 1, 2, [-2, -6, 0, 0, -4, 1, 0, 0]~]
S=[0, 0, 0, 0] rk(A_S)=0 A_S=[]
S=[0, 0, 0, 1] rk(A_S)=0 A_S=[]
S=[0, 0, 1, 0] rk(A_S)=0 A_S=[]
S=[0, 0, 1, 1] rk(A_S)=2 A_S=[1792160394037,13]
S=[0, 1, 0, 0] rk(A_S)=0 A_S=[]
S=[0, 1, 0, 1] rk(A_S)=2 A_S=[1792160394037,1792160394037]
S=[0, 1, 1, 0] rk(A_S)=2 A_S=[1792160394037,13]
S=[0, 1, 1, 1] rk(A_S)=4 A_S=[1792160394037,1792160394037,1792160394037,13]
S=[1, 0, 0, 0] rk(A_S)=0 A_S=[]
S=[1, 0, 0, 1] rk(A_S)=2 A_S=[1792160394037,13]
S=[1, 0, 1, 0] rk(A_S)=2 A_S=[1792160394037,1792160394037]
S=[1, 0, 1, 1] rk(A_S)=4 A_S=[1792160394037,1792160394037,1792160394037,13]
S=[1, 1, 0, 0] rk(A_S)=2 A_S=[1792160394037,13]
S=[1, 1, 0, 1] rk(A_S)=4 A_S=[1792160394037,1792160394037,1792160394037,13]
S=[1, 1, 1, 0] rk(A_S)=4 A_S=[1792160394037,1792160394037,1792160394037,13]
S=[1, 1, 1, 1] rk(A_S)=6 A_S=[1792160394037,1792160394037,1792160394037,1792160394037,1792160394037,13]
\end{verbatim}
\normalsize

\subsection{Experiments with the fields $K=\Q(\sqrt[p]{N})$, $N$ prime}
\smallskip
These fields are studied in great detail by Lecouturier in \cite[\S\,5]{Le}\,[Lec2018] for
their $p$-class groups and these fields have some remarkable
properties. For instance if ${\rm log}$ is the discrete logarithm for
$(\Z/p\Z)^\times$ provided with a primitive root $g$, the expression 
$T=\sm_{k=1}^{(N-1)/2} k\cdot {\rm log}(k) \pmod p$ governs,
under some conditions, the $p$-rank of $\Cl_K$ (from a result of 
Calegari--Emerton, after other similar results of Iimura, proved 
again in \cite[Theorem 1.1]{Le}\,[Lec2018]).

\smallskip
So we shall give the general caculations, for all $S \subseteq P$, 
with that of $T$. We assume $N$ prime congruent to $1$ modulo $p$,
but the reader may suppress this condition. It seems that many interesting
heuristics can be elaborated from the numerical results; we only give
some examples (recal that the structure of the class group is given
by the first data $S = \es$):

\footnotesize
\begin{verbatim}
{p=3;print("p=",p);n=8+floor(30/p);g=znprimroot(p);forprime(N=1,10^3,
if(Mod(N,p)!=1,next);P=x^p-N;print();print("P=",P);T=Mod(0,p);
for(k=1,(N-1)/2,if(Mod(k,p)==0,next);T=T+k*znlog(k,g));K=bnfinit(P,1);
F=idealfactor(K,p);d=component(matsize(F),1);F1=component(F,1);for(j=1,d,
print(component(F1,j)));for(z=2^d,2*2^d-1,bin=binary(z);mod=List;for(j=1,d,
listput(mod,component(bin,j+1),j));M=1;for(j=1,d,ch=component(mod,j);
if(ch==1,F1j=component(F1,j);ej=component(F1j,3);F1j=idealpow(K,F1j,ej);
M=idealmul(K,M,F1j)));Idn=idealpow(K,M,n);Kpn=bnrinit(K,Idn);
Hpn=component(component(Kpn,5),2);L=List;e=component(matsize(Hpn),2);R=0;
for(k=1,e,c=component(Hpn,e-k+1);w=valuation(c,p);if(w>0,R=R+1;
listinsert(L,p^w,1)));print("S=",mod," rk(A_S)=",R," A_S=",L)))}

p=3
P=x^3 - 7
[3, [-1, 1, 0]~, 3, 1, [1, 1, 1]~]
T=Mod(2,3) S=[0] rk(A_S)=1 A_S=[3]
T=Mod(2,3) S=[1] rk(A_S)=2 A_S=[387420489,387420489]
P=x^3 - 271
[3, [-2, 0, -1]~, 1, 1, [0, 0, 1]~]
[3, [-1, 1, 1]~, 2, 1, [2, 1, 0]~]
T=Mod(0,3) S=[0,0] rk(A_S)=1 A_S=[9]
T=Mod(0,3) S=[0,1] rk(A_S)=3 A_S=[129140163, 27, 3]
T=Mod(0,3) S=[1,0] rk(A_S)=2 A_S=[9, 3]
T=Mod(0,3) S=[1,1] rk(A_S)=4 A_S=[129140163,129140163, 27, 3]
P=x^3 - 523
[3, [0, 0, 1]~, 2, 1, [2, 1, 0]~]
[3, [1, 0, -1]~, 1, 1, [2, 1, 1]~]
T=Mod(0,3) S=[0,0] rk(A_S)=1 A_S=[9]
T=Mod(0,3) S=[0,1] rk(A_S)=2 A_S=[9, 3]
T=Mod(0,3) S=[1,0] rk(A_S)=3 A_S=[387420489, 9, 3]
T=Mod(0,3) S=[1,1] rk(A_S)=4 A_S=[387420489,129140163, 9, 3]

p=5
P=x^5 - 11
[5, [-1, 1, 0, 0, 0]~, 5, 1, [1, 1, 1, 1, 1]~]
T=Mod(4,5) S=[0] rk(A_S)=1 A_S=[5]
T=Mod(4,5) S=[1] rk(A_S)=3 A_S=[30517578125,6103515625,6103515625]
P=x^5 - 211
[5, [-1, 1, 0, 0, 0]~, 5, 1, [1, 1, 1, 1, 1]~]
T=Mod(4,5) S=[0] rk(A_S)=3 A_S=[5, 5, 5]
T=Mod(4,5) S=[1] rk(A_S)=5 A_S=[6103515625,6103515625,6103515625, 5, 5]
P=x^5 - 401
[5, [-1, 1, 0, 1, 0]~, 4, 1, [4, 3, 2, 0, 1]~]
[5, [1, 0, 0, -1, 0]~, 1, 1, [4, 3, 2, 1, 1]~]
T=Mod(0,5) S=[0,0] rk(A_S)=2 A_S=[5, 5]
T=Mod(0,5) S=[0,1] rk(A_S)=2 A_S=[25, 5]
T=Mod(0,5) S=[1,0] rk(A_S)=3 A_S=[6103515625,6103515625, 25]
T=Mod(0,5) S=[1,1] rk(A_S)=4 A_S=[6103515625,6103515625,1220703125, 25]

p=7
P=x^7 - 29
[7, [-1, 1, 0, 0, 0, 0, 0]~, 7, 1, [1, 1, 1, 1, 1, 1, 1]~]
T=Mod(6,7) S=[0] rk(A_S)=1 A_S=[7]
T=Mod(6,7) S=[1] rk(A_S)=4 A_S=[96889010407,13841287201,13841287201,13841287201]
P=x^7 - 197
[7, [0, 0, 0, 0, 0, 0, 1]~, 1, 1, [6, 5, 4, 3, 3, 2, 1]~]
[7, [1, 0, 0, 0, 0, 0, -1]~, 6, 1, [6, 5, 4, 3, 1, 2, 1]~]
T=Mod(0,7) S=[0,0] rk(A_S)=1 A_S=[7]
T=Mod(0,7) S=[0,1] rk(A_S)=4 A_S=[96889010407,13841287201, 1977326743, 49]
T=Mod(0,7) S=[1,0] rk(A_S)=1 A_S=[7]
T=Mod(0,7) S=[1,1] rk(A_S)=5 A_S=[96889010407,13841287201,1977326743,1977326743, 49]
P=x^7 - 337
[7, [-1, 1, 0, 0, 0, 0, 0]~, 7, 1, [1, 1, 1, 1, 1, 1, 1]~]
T=Mod(2,7) S=[0] rk(A_S)=2 A_S=[7, 7]
T=Mod(2,7) S=[1] rk(A_S)=5 A_S=[13841287201,13841287201,13841287201,13841287201, 7]

p=11
P=x^11 - 67
[11, [-1, 1, 0, 0, 0, 0, 0, 0, 0, 0, 0]~, 11, 1, [1, 1, 1, 1, 1, 1, 1, 1, 1, 1, 1]~]
T=Mod(8,11) S=[0] rk(A_S)=2 A_S=[11, 11]
T=Mod(8,11) S=[1] rk(A_S)=7 A_S=[285311670611,285311670611,25937424601,
                                                           25937424601,25937424601,25937424601, 11]
P=x^11 - 727
[11, [-5, 0, 0, 0, 0, 0, 0, 0, 0, 0, -5]~, 1, 1, [10, 9, 8, 7, 6, 5, 4, 6, 3, 2, 1]~]
[11, [-5, 0, 0, 0, 0, 0, 0, 0, 0, 0, 5]~, 10, 1, [10, 9, 8, 7, 6, 5, 4, 4, 3, 2, 1]~]
T=Mod(0,11) S=[0,0] rk(A_S)=1 A_S=[11]
T=Mod(0,11) S=[0,1] rk(A_S)=6 A_S=[25937424601,25937424601,25937424601,25937424601,2357947691, 121]
T=Mod(0,11) S=[1,0] rk(A_S)=1 A_S=[11]
T=Mod(0,11) S=[1,1] rk(A_S)=7 A_S=[25937424601,25937424601,25937424601,
                                                             25937424601,2357947691,2357947691,121]
p=13
P=x^13 - 53
[13, [-1, 1, 0, 0, 0, 0, 0, 0, 0, 0, 0, 0, 0]~, 13, 1, [1, 1, 1, 1, 1, 1, 1, 1, 1, 1, 1, 1, 1]~]
T=Mod(11,13) S=[0] rk(A_S)=1 A_S=[13]
T=Mod(11,13) S=[1] rk(A_S)=7 A_S=[1792160394037,137858491849,137858491849,
                                               137858491849,137858491849,137858491849,137858491849]
P=x^13 - 677
[13, [-4, 0, 0, 0, 0, 0, 0, 0, 0, 0, 0, 0, 4]~, 12, 1, [12, 11, 10, 9, 8, 7, 6, 5, 5, 4, 3, 2, 1]~]
[13, [5, 0, 0, 0, 0, 0, 0, 0, 0, 0, 0, 0, -4]~, 1, 1,[12, 11, 10, 9, 8, 7, 6, 5, 2, 4, 3, 2, 1]~]
T=Mod(0,13) S=[0,0] rk(A_S)=1 A_S=[13]
T=Mod(0,13) S=[0,1] rk(A_S)=1 A_S=[13]
T=Mod(0,13) S=[1,0] rk(A_S)=7 A_S=[137858491849,137858491849,137858491849,
                                                        137858491849,137858491849,10604499373, 169]
T=Mod(0,13) S=[1,1] rk(A_S)=8 A_S=[137858491849,137858491849,137858491849,
                                            137858491849,137858491849,10604499373,10604499373, 169]
\end{verbatim}
\normalsize

\subsection{The fields $K=\Q \big(\sqrt{- \sqrt{-q}}\big)$ associated to elliptic curves}
\smallskip
These fields, used by Coates--Li in \cite{CL,CL2}\,[CL2018-2019] to prove non-vanishing theorems 
for the central values at $s = 1$ of the complex $L$-series of a family of elliptic curves
studied by Gross (for any prime $q \equiv 7 \pmod {8}$ and $p=2$), 
are particularly interesting.

\smallskip
Note once for all that the signature of $K$ is ${\sf [0,2]}$, the Galois closure 
of $K$ is of degree $8$ with Galois group ${\sf  [8, -1, 1, "D(4)"]}$ and
$D_K= 2^m\,q^3$.

\subsubsection{Program for various $p$}
\smallskip
In this part, we fix the prime number $q$ and compute the structure of 
${\mathcal A}_{K,S}$ for all sets $S \subseteq P$.
Recall that the parameter $n$ must be such that $p^n$ be much 
larger than the exponent of ${\mathcal T}_K$. 

\smallskip
For instance, for $P=x^4+23$, we give the 
results for $p=3$ and $p=71$:

\smallskip
\footnotesize
\begin{verbatim}
{q=23;P=x^4+q;print("P=",P);bp=2;Bp=500;n0=8;K=bnfinit(P,1);
forprime(p=bp,Bp,n=n0+floor(30/p);print();print("p=",p);F=idealfactor(K,p);
d=component(matsize(F),1);F1=component(F,1);for(j=1,d,
print(component(F1,j)));for(z=2^d,2*2^d-1,bin=binary(z);mod=List;for(j=1,d,
listput(mod,component(bin,j+1),j));M=1;for(j=1,d,ch=component(mod,j);
if(ch==1,F1j=component(F1,j);ej=component(F1j,3);FF1j=idealpow(K,F1j,ej);
M=idealmul(K,M, FF1j)));Idn=idealpow(K,M,n);Kpn=bnrinit(K,Idn);
Hpn=component(component(Kpn,5),2);L=List;e=component(matsize(Hpn),2);R=0;
for(k=1,e,c=component(Hpn,e-k+1);w=valuation(c,p);if(w>0,R=R+1;
listinsert(L,p^w,1)));print("S=",mod," rk(A_S)=",R," A_S=",L)))}

P=x^4 + 23
p=3
[3, [-1, 1, 0, 0]~, 1, 1, [1, 0, 1, 1]~]
[3, [1, 1, 0, 0]~, 1, 1, [0, 0, 0, 1]~]
[3, [2, 0, 2, 0]~, 1, 2, [0, 0, -1, 0]~]
S=[0, 0, 0] rk(A_S)=1 A_S=[3]
S=[0, 0, 1] rk(A_S)=1 A_S=[68630377364883]
S=[0, 1, 0] rk(A_S)=1 A_S=[3]
S=[0, 1, 1] rk(A_S)=2 A_S=[68630377364883, 22876792454961]
S=[1, 0, 0] rk(A_S)=1 A_S=[3]
S=[1, 0, 1] rk(A_S)=2 A_S=[68630377364883, 22876792454961]
S=[1, 1, 0] rk(A_S)=1 A_S=[68630377364883]
S=[1, 1, 1] rk(A_S)=3 A_S=[68630377364883, 22876792454961, 22876792454961]
p=71
[71, [-32, 1, 0, 0]~, 1, 1, [0, 29, -5, 4]~]
[71, [32, 1, 0, 0]~, 1, 1, [4, 29, 9, 4]~]
[71, [31, 0, 2, 0]~, 1, 2, [-29, 0, 2, 0]~]
S=[0, 0, 0] rk(A_S)=0 A_S=[]
S=[0, 0, 1] rk(A_S)=1 A_S=[9095120158391]
S=[0, 1, 0] rk(A_S)=1 A_S=[71]
S=[0, 1, 1] rk(A_S)=2 A_S=[9095120158391, 9095120158391]
S=[1, 0, 0] rk(A_S)=1 A_S=[71]
S=[1, 0, 1] rk(A_S)=2 A_S=[9095120158391, 9095120158391]
S=[1, 1, 0] rk(A_S)=2 A_S=[9095120158391, 71]
S=[1, 1, 1] rk(A_S)=3 A_S=[9095120158391, 9095120158391, 9095120158391]
\end{verbatim}
\normalsize

\smallskip
The user is invited to vary $n$ at will to certify the numerical results when 
the $p$-rank of ${\mathcal A}_{K,S}$ is unknown (i.e., when $S \subsetneq P$).
In the above examples, some ${\mathcal T}_{K,S}$ are of order $p$ and the 
$\Z_p$-rank of ${\mathcal A}_{K,S}$ is $0$ or $1$.

\subsubsection{Program for various $q$ and $p=2$}
\smallskip
The analogous program is the following ($n=32$ is large enough):

\footnotesize
\begin{verbatim}
{bq=3;Bq=100;p=2;n=32;forprime(q=bq,Bq,P=x^4+q;print();
print("q=",q," ",Mod(q,16));K=bnfinit(P,1);F=idealfactor(K,p);
d=component(matsize(F),1);F1=component(F,1);for(j=1,d,
print(component(F1,j)));for(z=2^d,2*2^d-1,bin=binary(z);mod=List;for(j=1,d,
listput(mod,component(bin,j+1),j));M=1;for(j=1,d,ch=component(mod,j);
if(ch==1,F1j=component(F1,j);ej=component(F1j,3);F1j=idealpow(K,F1j,ej);
M=idealmul(K,M,F1j)));Idn=idealpow(K,M,n);Kpn=bnrinit(K,Idn);
Hpn=component(component(Kpn,5),2);L=List;e=component(matsize(Hpn),2);R=0;
for(k=1,e,c=component(Hpn,e-k+1);w=valuation(c,p);if(w>0,R=R+1;
listinsert(L,p^w,1)));print("S=",mod," rk(A_S)=",R," A_S=",L)))}
\end{verbatim}
\normalsize

We give an example of each congruence class $q \!\!\pmod {16}$; 
for $q \equiv 7 \pmod {16}$, the decomposition of $(2)$ in $\Q(\sqrt{-q})$
is $(2) = {\mathfrak p} \cdot {\mathfrak p}^*$ where $e_{\mathfrak p}=2$
in $K/\Q$:

\smallskip
\footnotesize
\begin{verbatim}
q=17    Mod(1, 16)
[2, [1, 1, 0, 0]~, 4, 1, [1, 1, 1, 1]~]
S=[0] rk(A_S)=2 A_S=[8, 2]
S=[1] rk(A_S)=5 A_S=[4294967296, 2147483648, 2147483648, 8, 2]

q=3    Mod(3, 16)
[2, [1, 0, -1, 0]~, 2, 2, [1, 0, 1, 0]~]
S=[0] rk(A_S)=0 A_S=[]
S=[1] rk(A_S)=3 A_S=[4294967296, 2147483648, 1073741824]

q=5    Mod(5, 16)
[2, [1, 1, 0, 0]~, 4, 1, [1, 1, 1, 1]~]
S=[0] rk(A_S)=1 A_S=[4]
S=[1] rk(A_S)=3 A_S=[8589934592, 4294967296, 4294967296]

q=7    Mod(7, 16)
[2, [0, -1, 0, 1]~, 2, 1, [1, 0, 0, 1]~]
[2, [0, 1, 0, 0]~, 1, 2, [1, 1, 0, 0]~]
S=[0, 0] rk(A_S)=0 A_S=[]
S=[0, 1] rk(A_S)=2 A_S=[1073741824, 4]
S=[1, 0] rk(A_S)=1 A_S=[2147483648]
S=[1, 1] rk(A_S)=4 A_S=[2147483648, 2147483648, 1073741824, 2]

q=41    Mod(9, 16)
[2, [1, 1, 0, 0]~, 4, 1, [1, 1, 1, 1]~]
S=[0] rk(A_S)=2 A_S=[16, 2]
S=[1] rk(A_S)=4 A_S=[8589934592, 4294967296, 2147483648, 8]

q=11    Mod(11, 16)
[2, [1, 0, -1, 0]~, 2, 2, [1, 0, 1, 0]~]
S=[0] rk(A_S)=0 A_S=[]
S=[1] rk(A_S)=3 A_S=[4294967296, 2147483648, 1073741824]

q=13    Mod(13, 16)
[2, [1, 1, 0, 0]~, 4, 1, [1, 1, 1, 1]~]
S=[0] rk(A_S)=1 A_S=[4]
S=[1] rk(A_S)=3 A_S=[8589934592, 4294967296, 4294967296]

q=31    Mod(15, 16)
[2, [-1, 0, 0, 1]~, 1, 1, [0, 0, 0, 1]~]
[2, [0, 1, -1, 0]~, 2, 1, [1, 1, 0, 0]~]
[2, [2, 0, 1, 1]~, 1, 1, [1, 0, 1, 1]~]
S=[0, 0, 0] rk(A_S)=0 A_S=[]
S=[0, 0, 1] rk(A_S)=1 A_S=[4]
S=[0, 1, 0] rk(A_S)=2 A_S=[2147483648, 4]
S=[0, 1, 1] rk(A_S)=3 A_S=[2147483648, 1073741824, 8]
S=[1, 0, 0] rk(A_S)=1 A_S=[4]
S=[1, 0, 1] rk(A_S)=3 A_S=[1073741824, 4, 2]
S=[1, 1, 0] rk(A_S)=3 A_S=[2147483648, 1073741824, 8]
S=[1, 1, 1] rk(A_S)=5 A_S=[2147483648, 1073741824, 1073741824, 8, 2]
\end{verbatim}
\normalsize

\begin{remark} A more complete table shows some rules:

\smallskip
\quad (i) For $q \equiv 3 \pmod 8$, ${\mathcal T}_{K,S}=1$ 
for $S=\es$ and $S = P = \{{\mathfrak p}\}$;

\smallskip
\quad (ii) For $q \equiv 5 \pmod 8$, ${\mathcal T}_{K,\es} = \Cl_K \simeq \Z/4\Z$ 
and ${\mathcal T}_{K,P}=1$ for $P = \{{\mathfrak p}\}$ (which means that 
the $2$-Hilbert class field of $K$ is contained in the compositum of the 
$\Z_2$-extensions of $K$);

\smallskip
\quad (iii) For $q \equiv 7 \pmod {16}$, for $S=\{{\mathfrak p}\}$
with $e_{\mathfrak p}=2$, we get ${\mathcal T}_{K,S} \simeq \Z/4\Z$ and
for  $S=\{{\mathfrak p}^*\}$ with $e_{{\mathfrak p}^*}=1$, we get ${\mathcal T}_{K,S}=1$;
then ${\mathcal T}_{K,P} \simeq \Z/2\Z$.

\smallskip
These properties may be proved easily and are left to 
the reader as exercises on the ${\rm Log}_S^{}$-function (Definition \ref{4}): consider 
first the arithmetic of the subfield $k=\Q(\sqrt{-q})$ and use fixed point formulas 
\eqref{fix} in $K/k$.

\smallskip
\quad (iv) For $q \equiv 15 \pmod {16}$, the results do not follow any obvious 
rule and offers some interesting examples as the following ones:

\footnotesize
\begin{verbatim}
q=5503 
[2, [-1, 0, 0, 1]~, 1, 1, [0, 0, 0, 1]~]
[2, [0, 1, -1, 0]~, 2, 1, [1, 1, 0, 0]~]
[2, [2, 0, 1, 1]~, 1, 1, [1, 0, 1, 1]~]
S=[0, 0, 0] rk(A_S)=0 A_S=[]
S=[0, 0, 1] rk(A_S)=1 A_S=[512]
S=[0, 1, 0] rk(A_S)=2 A_S=[2147483648, 8]
S=[0, 1, 1] rk(A_S)=3 A_S=[2147483648, 1073741824, 16]
S=[1, 0, 0] rk(A_S)=1 A_S=[512]
S=[1, 0, 1] rk(A_S)=3 A_S=[1073741824, 512, 2]
S=[1, 1, 0] rk(A_S)=3 A_S=[2147483648, 1073741824, 16]
S=[1, 1, 1] rk(A_S)=5 A_S=[2147483648, 1073741824, 1073741824, 16, 2]

q=8191 
[2, [-1, 0, 0, 1]~, 1, 1, [0, 0, 0, 1]~]
[2, [0, 1, -1, 0]~, 2, 1, [1, 1, 0, 0]~]
[2, [2, 0, 1, 1]~, 1, 1, [1, 0, 1, 1]~]
S=[0, 0, 0] rk(A_S)=0 A_S=[]
S=[0, 0, 1] rk(A_S)=1 A_S=[64]
S=[0, 1, 0] rk(A_S)=2 A_S=[2147483648, 64]
S=[0, 1, 1] rk(A_S)=3 A_S=[2147483648, 1073741824, 128]
S=[1, 0, 0] rk(A_S)=1 A_S=[64]
S=[1, 0, 1] rk(A_S)=3 A_S=[1073741824, 64, 2]
S=[1, 1, 0] rk(A_S)=3 A_S=[2147483648, 1073741824, 128]
S=[1, 1, 1] rk(A_S)=5 A_S=[2147483648, 1073741824, 1073741824, 128, 2]

q=123551 
[2, [-1, 0, 0, 1]~, 1, 1, [0, 0, 0, 1]~]
[2, [0, 1, -1, 0]~, 2, 1, [1, 1, 0, 0]~]
[2, [2, 0, 1, 1]~, 1, 1, [1, 0, 1, 1]~]
S=List([0, 0, 0]) rk(A_S)=0 A_S=List([])
S=List([0, 0, 1]) rk(A_S)=1 A_S=List([16])
S=List([0, 1, 0]) rk(A_S)=2 A_S=List([2147483648, 16])
S=List([0, 1, 1]) rk(A_S)=3 A_S=List([2147483648, 1073741824, 32])
S=List([1, 0, 0]) rk(A_S)=1 A_S=List([16])
S=List([1, 0, 1]) rk(A_S)=3 A_S=List([1073741824, 16, 2])
S=List([1, 1, 0]) rk(A_S)=3 A_S=List([2147483648, 1073741824, 32])
S=List([1, 1, 1]) rk(A_S)=5 A_S=List([2147483648, 1073741824, 1073741824, 32, 2])
\end{verbatim}
\normalsize
\end{remark}

\appendix

\section{Appendix: History of abelian $p$-ramification}

\subsection{Motivations}
\smallskip
We intend, in this detailed survey, to give a maximum of practical 
information and results about the torsion groups ${\mathcal T}_{K,S}$ that we have 
numerically computed in the first part of the paper with a PARI/GP program.
Since all the invariants, associated with ${\mathcal T}_{K,S}$, need numerical 
computations for a better understanding, we choose the more suitable technical 
presentation (the main philosophcal remark is that {\it they are all equivalent}).

\smallskip
For convenience, we indicate both the original historical contributions 
and the corresponding results processed systematically in our book \cite{Gr3}\,[Gr2003].

\smallskip
We will not detail the immense domains of pro-$p$-groups and Galois 
cohomology, whose main purpose is for instance the existence of infinite towers 
of $S$-ramified extensions and the Fontaine--Mazur conjecture studied 
by various schools of mathematicians (for this, see, e.g., \cite[\S\,10]{NSW}\,[NSW2000]), 
nor the analytic aspects as the non-vanishing at $s=1$ of complex $L$-series 
associated to elliptic curves\,$\ldots$
Similarly, we shall not consider the context of Iwasawa's theory 
because this efficient tool does not exempt from having the ``basic'' 
arithmetical properties of the corresponding objects. 

\smallskip
Note that the solutions of the analogous problems of $S$-ramification 
over local fields are not sufficient for a ``globalization'' over a number field 
$K$ as remarqued by Nguyen Quang Do in \cite[\S\,9]{Ng1}\,[Nqd1982]. Indeed, the global
theory depends on Leopoldt's conjecture (usually assumed) and the torsion
groups ${\mathcal T}_{K,S}$ are, in some sense, refinements of this 
conjecture.

\smallskip
So we will focus, mainly, on class field theory and on these specific deep 
$p$-adic properties or conjectures  which are, in our opinion, the main 
obstructions for many contemporary researches.

\smallskip
We will not give the most general statements but restrict ourselves to the 
case of $S$-ramification, $S \subseteq P$, whithout decomposition of
finite or infinite places (indeed, in these more elaborate 
cases, the formalism is identical and may be found in our book). Since
the properties of $S$-ramification may be used by many researchers 
working on different subjects, we will try to explain the numerous 
steps of its progress. This must be understood for practical information and 
will be an opportunity to clarify the vocabulary and the main contributions.

\smallskip
We apologize for the probable lack of references (and citation of their authors). 

\subsection{Prehistory}\label{story}
\smallskip
The origin of interest for $S$-ramification theory over a number field is probably 
a paper of Brumer \cite{Br}\,[Bru1966], following Serre's book \cite{Se1}\,[Ser1964] 
and seems also due to a lecture by \v Safarevi\v c (1963) showing the importance 
of the subject.
In \cite{Sa}\,[Sha1964], \v Safarevi\v c gives the cohomological characteristics 
of the group ${\mathcal G}_{K,S}$ (number of generators and relations, 
cohomological dimension\,$\ldots$). 

\smallskip
Recall at this step the Golod--\v Safarevi\v c
theorem (1964), named soon after the theorem of 
Golod--\v Safarevi\v c--Gasch\"utz--Vinberg, saying that if a pro-$p$-group
${\mathcal G}$ is finite, then $r({\mathcal G}) > \frac{1}{4}\,(d({\mathcal G}))^2$
where $d({\mathcal G})$ (resp. $r({\mathcal G})$) is the minimal number
of generators (resp. relations) for the presentation of ${\mathcal G}$.
All of this was developed in Koch's book \cite{Ko}\,[Koch1970]
from the works of many German mathematicians and is amply
improved in \cite{NSW}\,[NSW2000]
(see also in Hajir--Maire \cite{HM0,HM00}\,[HM2001-2002a] a good introduction on the subject 
and some of its developments \cite{HM3}\,[HM2002b], \cite{M2,M3}\,[Mai2010-2018], 
\cite{HM2,HM1}\,[HM2018a-2018b]).

\smallskip
More precisely, in \cite[Th\'eor\`eme I]{Sa}\,[Sha1964], \v Safarevi\v c gives,
for any number field $K$ and any set of places $S$, the main 
formula \eqref{cha} that we recall:

\subsubsection{\v Safarevi\v c formula}\label{00}
\smallskip
The $p$-rank of the $\Z_p$-module ${\mathcal A}_{K,S}$
(giving the minimal number of generators
${\rm dim}_{\F_p}({\rm H}^1({\mathcal G}_{K,S},\Z/p\Z))$
of ${\mathcal G}_{K,S}$) is:
\begin{equation}\label{chaS}
{\rm rk}_p({\mathcal A}_{K,S})  =  {\rm rk}_p \big (V_{K,S}/K_{(S)}^{\times p} \big) 
 + \sm_{{\mathfrak p} \in S \,\cap\, P} [K_{\mathfrak p}  : \Q_p] 
 + \sm_{{\mathfrak p} \in S} \delta_{\mathfrak p} - \delta_K - ( r_1 + r_2 -1) , 
\end{equation}

\noindent
where $K_{(S)}^\times :=  \big \{ \alpha \in K^\times,
\hbox{$\alpha$ prime to $S$} \big\}$,
$V_{K,S} := \big \{ \alpha \in K_{(S)}^{\times},\  (\alpha) = {\mathfrak a}^p \big\}$, 
then $\delta_{\mathfrak p} = 1$ or $0$ according as the 
completion $K_{\mathfrak p}$ contains $\mu_p$ or not, and $\delta_K = 1$ or $0$ 
according as $K$ contains $\mu_p$ or not.

\medskip
Of course, ${\rm dim}_{\F_p}({\rm H}^2({\mathcal G}_{K,S},\Z/p\Z))$, giving
the minimal number of relations, is easily obtained only when $P \subseteq S$ 
(equal to ${\rm rk}_p({\mathcal T}_{K,S})$ under Leopoldt's conjecture), 
which shall explain the forthcoming studies about this: 

\smallskip
\cite{NSW}\,[NSW2000], \cite{Win1,Win2}\,[Win1989-1991], \cite{Ya}\,[Yam1993], 
\cite{M1}\,[Mai2005], \cite{Lab}\,[Lab2006], \cite{Vog}\,[Vog2007], \cite{Ko}\,[Koch1970], 
\cite{Neum}\,[Neu1976], Haberland \cite{Hab}\,[Hab1978], 
\cite{Sch}\,[Sch2010], El Habibi--Ziane \cite{ElHZ}\,[ElHZ2018]\,$\ldots$.

\subsubsection{Kubota formalism}
\smallskip
Mention that Kubota \cite{Kub}\,[Kub1957] begins  the study of the structure of the
dual ${\mathcal A}_{K,S}^*$ of ${\mathcal A}_{K,S}$, study which is based 
on the Grunwald--Wang theorem and which leads to a characterization of this 
group in terms of its fundamental invariants called, following Kaplansky, the
``Ulm invariants''. 

\smallskip
Then in \cite{Mi}\,[Miki1978], Miki uses this formalism, about $\ell(=p)$-ramification,
then class field theory, Iwasawa's theory, in direction of Leopoldt's conjecture. 
Some statements, equivalent to some results that we shall recall in this survey 
(as well as the notion of $p$-rationality and its main properties), should be mentioned 
in his paper, despite the difficulty of translating vocabulary and technique.

\subsection{Main developments after the pioneering works}
\smallskip
The computation of ${\rm rk}_p({\mathcal T}_{K,P})$, from Kummer 
theory, is already given by Bertrandias--Payan \cite{BP}\,[BP1972], then in 
\cite[Th\'eor\`emes I.2, I.3, Corollaire 1]{Gr01}\,[Gr1982] and by many 
authors, for instance by means of cohomological techniques (e.g., 
\cite[Proposition 3]{Mo2}\,[Mov1990]). 

\smallskip
This will give reflection formulas.
 
\subsubsection{Reflection and rank formulas}\label{0}
\smallskip
From \cite[Chapitre III, \S\,10]{Gr2}\,[Gr1998] or \cite[\S\,II.5.4.1]{Gr3}[Gr2003].
Using the \v Safarevi\v c formula and Kummer theory when $K$ contains
the group $\mu_p$ of $p$th roots of unity, and writing (for $S \subseteq P$):
$$\hbox{$P=S \,\cup\, \Sigma\ $ with $\ S \,\cap\, \Sigma = \es$,} $$

\noindent
one obtains the reflection theorem in its simplest form:
\begin{equation}\label{rf}
{\rm rk}_p({\mathcal A}_{K,S}^{\Sigma}) -
{\rm rk}_p({\mathcal A}_{K, \Sigma}^{S\,\rm res}) = \order S - \order \Sigma
+ \sm_{{\mathfrak p} \in S}\, [K_{\mathfrak p} : \Q_p] - r_1 - r_2, 
\end{equation}

\noindent
where ${\mathcal A}_{K,S}^{\Sigma}$ is the Galois group of the maximal
abelian pro-$p$-extension of $K$ in $H_{K,S}$, which is 
$\Sigma \,\cup\, \{\infty\}$-split 
(i.e., in which all the places of $\Sigma \,\cup\,\{\infty\}$ split completely), and
similarly for the definition of ${\mathcal A}_{K, \Sigma}^{S\rm res}$, in the 
restricted sense for $p=2$ (i.e., only $S$-split); in other words, the mention of 
$\{\infty\}$ is implicit in the upper script to give the ordinary sense when $p=2$.

\smallskip
The case $S=P$ leads to the following well-known result:

\begin{theorem} \label{thmf} \label{tp=1} \cite[Proposition III.4.2.2]{Gr3}\,[Gr2003].
Let $K$ be any number field fulfilling the Leopoldt conjecture
for the prime number $p$. Let $K':=K(\mu_p)$, $P'$ be the set of
$p$-places above $P$ in $K'$, and let $P^{\rm dec}$ be
the set of $p$-places of $K$ totaly split in $K'$. 
Let $\omega$ be the Teichm\"uller character and denote by ${\rm rk}_\omega$ 
the $p$-rank of an isotypic $\omega$-component for ${\rm Gal}(K'/K)$; then:
$${\rm rk}_p({\mathcal T}_{K,P}) = {\rm rk}_\omega(\Cl^{P' \rm res}_{K'})
+ \order P^{\rm dec} - \delta_K, $$

\noindent
where $\Cl_{K'}^{P'\rm res}$ is the quotient of the $p$-class group $\Cl_{K'}^{\rm res}$
by the subgroup generated by the classes of $P'$ (in the restricted sense for $p=2$) and where
$\delta_K = 1$ or $0$ according as $K$ contains $\mu_p$ or not. Whence the following properties:

\smallskip
\quad (i) If $\mu_p \subset K$, we then have
${\rm rk}_p({\mathcal T}_{K,P}) = {\rm rk}_p(\Cl_{K}^{P\,\rm res}) + \order P - 1$.

\smallskip
\quad (ii) We have ${\mathcal T}_{K,P}=1$ if and only if:

\smallskip
\qquad $\bullet$ $\mu_p \not\subset K$: then $P^{\rm dec}=\es$ 
and the $\omega$-component of $\Cl_{K}^{\rm res}$ is trivial;

\smallskip
\qquad $\bullet$  $\mu_p \subset K$: $p$ does not split in $K/\Q$
and the unique ${\mathfrak p} \in P$ generates~$\Cl_{K}^{\rm res}$.
\end{theorem}

\begin{example}\label{A1}
For $K=\Q(\mu_p)=:\Q(\zeta_p)$, $p\ne 2$, taking 
$\Sigma=\es$ and $S=P$:
$${\rm rk}_p({\mathcal A}_{K,P}) -
{\rm rk}_p({\mathcal A}_{K, \es}^P) = 1 +p-1 - 
\hbox{$\frac{p-1}{2} = \frac{p+1}{2} $}.$$

Since ${\mathcal A}_{K, \es}^P = \Cl_K/ \langle\cl_K({\mathfrak p}) \rangle$, with
${\mathfrak p}=(1-\zeta_p)$, and ${\mathcal A}_{K,P} \simeq 
\Z_p^{\frac{p+1}{2}} \plus {\mathcal T}_{K,P}$,
this yields:
\begin{equation}\label{reflection}
{\rm rk}_p({\mathcal T}_{K,P}) = {\rm rk}_p(\Cl_K),
\end{equation}

\noindent
as well as the writing ${\rm rk}_p({\mathcal T}_{K,P}^\pm) 
= {\rm rk}_p(\Cl_K^\mp)$ (for analogous equalities with pairs of isotopic 
components associated by means of the mirror involution, and the 
consequences for Vandiver's conjecture, see \cite{Gr11}\,[Gr2019b]).
\end{example}

If the condition $S \,\cup\, \Sigma = P$ is not fulfilled,
we have (still assuming $\mu_p \subset K$) the reflection formula:
\begin{equation}\label{rf*}
{\rm rk}_p({\mathcal A}_{K,S}^{\Sigma}) -
{\rm rk}_p(\Cl_K^{S \rm res}({\mathfrak m}^*) )=
 \order S - \order \Sigma
+ \sm_{{\mathfrak p} \in S} [K_{\mathfrak p} : \Q_p] - r_1 - r_2,\ \ 
\hbox{with ${\mathfrak m}^* := \prd_{{\mathfrak p} \in \Sigma} 
{\mathfrak p}^{p e_{\mathfrak p}+1} \cdot 
\prd_{{\mathfrak p} \in P \setminus S \,\cup\,\Sigma} 
{\mathfrak p}^{p e_{\mathfrak p}}$}
\end{equation}

\noindent
where $\Cl_K^{S \rm res}({\mathfrak m}^*)$ is the  $S$-split
$p$-ray class group of modulus ${\mathfrak m}^*$
(see \cite[Exercise II.5.4.1, proof of (iii)]{Gr3}\,[Gr2003] and (iv) for the case $p=2$).
Note that $\Cl_K^{S \rm res}({\mathfrak m}^*)$ is isomorphic to
a quotient of ${\mathcal A}_{K,P \setminus S}^{S \rm res}$.

\smallskip
Finaly, if $K$ does not contain $\mu_p$, but assuming 
$P=S \,\cup\, \Sigma$ with $S \,\cap\, \Sigma = \es$,
the general formula is:
\begin{equation}\label{rfomega}
{\rm rk}_p({\mathcal T}_{K,S}^{\Sigma}) =
{\rm rk}_\omega({\mathcal A}_{K', \Sigma'}^{S'\,\rm res}) 
+ \sm_{{\mathfrak p} \in S} \delta_{\mathfrak p}
-  \delta_K - \order \Sigma
- \big(r_1 + r_2 - 1 - r_{K,S}^\Sigma \big),
\end{equation}

\noindent
where
$r_{K,S}^\Sigma= \sm_{{\mathfrak p} \in S}[K_{\mathfrak p} : \Q_p]-
\wt r_{K,S}^\Sigma$; here, $\wt r_{K,S}^\Sigma \leq r_2+1$ is the 
$\Z_p$-rank of $\Z_p {\rm log}_S(I_{K,S})$
modulo $\Q_p  {\rm log}_S (E_K^\Sigma)$
dealing with the group $E_K^\Sigma$ of $\Sigma$-units of $K$
(see also \cite{M1}\,[Mai2005], \cite{Vog}\,[Vog2007] for some applications).

\smallskip
One can restrict some of the above equalities to $p$-class groups, 
giving only inequalities on the $p$-ranks (Hecke theorem (1910), Scholz 
theorem (1932), Leopoldt Spiegelungssatz (1958), 
Armitage--Fr\"ohlich--Serre, Oriat, for $p=2$.

\smallskip
For reflection theorems and formulas with characters,
see \cite[II.5.4, Theorem II.5.4.5)]{Gr3}\,[Gr2003] from the computations of
\cite[Ch. I, Theorem 5.18]{Gr2}\,[Gr1998]
where $p$-rank formulas link $p$-class groups and torsion groups 
as in Theorem \ref{tp=1} (this context is used by Ellenberg--Venkatesh in \cite{EV}\,[EV2007]
for the $\varepsilon$-conjecture on $p$-class groups). 

\smallskip
For the annihilation of the Galois module ${\mathcal T}_{K,P}$, of real 
abelian extensions $K/\Q$, in relation with the construction of  $p$-adic 
$L$-functions and reflection principle, see \cite{Gr12}\,[Gr2018c] and its bibliography. 
There is probably {\it equivalent information} whatever the process 
(algebraic or analytic), as shown by Oriat in \cite{Or}\,[Ori1986].
This logical aspect should deserve further investigation.

\subsubsection{Regulators and $p$-adic residues of the $\zeta_p$-functions} \label{1}
\smallskip
We continue the story with the $p$-adic analytic computations of the residue
of the $p$-adic $\zeta$-function at $s=1$ of real abelian fields $K$ by
Amice--Fresnel \cite{AF}\,[AF1972], from Kubota--Leopoldt $L_p$-functions (1964),
by Coates \cite{Co}\,[Coa1977], Serre \cite{Se2}\,[Ser1978] introducing $p$-adic 
pseudo-measures, then by Colmez \cite{Col}\,[Col1988] in full generality, via the formula
$\frac{1}{2^{[K : \Q]-1}} \ds \,\lim_{s \to 1} (s-1)\, \zeta_{K,p}(s) = 
\frac{R_p\, h\, E_p(1)}{\sqrt D}$, 
where $R_p$ is the classical $p$-adic regulator, $h$ the class number, $D$
the discriminant of $K$ and $E_p(1)$ the eulerian factor 
$\prod_{{\mathfrak p} \mid p} (1-{\rm N}{\mathfrak p}^{-1})$.
For totally real fields, the normalised $p$-adic regulator 
${\mathcal R}_{K,P}$, in the formula \eqref{partial},
is given (under Leopoldt's conjecture) by the expression
\cite[Proposition 5.2]{Gr7}\,[Gr2018a]:
$$\order {\mathcal R}_{K,P} \sim \frac{1}{2} \cdot
\frac{\big(\Z_p : {\rm log}({\rm N}_{K/\Q}(U_{K,P})) \big)}
{ \order {\mathcal W}_{K,P} \cdot \prod_{{\mathfrak p} \mid p}{\rm N} {\mathfrak p}}
\cdot \frac {R_p}{\sqrt {D}}, $$

\noindent
where $\sim$ means equality up to a $p$-adic unit factor; whence:
$$ \hbox{$\frac{1}{2^{[K : \Q]-1}}$}\lim_{s \to 1} (s-1)\, \zeta_{K,p}(s) = 
\hbox{$\frac{1}{p\,[K \,\cap\, \Q^{\rm c} : \Q]}$} \, \order{\mathcal T}_{K,P}, $$

\noindent
where $ \Q^{\rm c}$ is the $\Z_p$-cyclotomic extension of $K$.
In \cite{Hat1}\,[Hat1987], Hatada uses the link between the $p$-adic valuation of
$\zeta_K(2-p)$ and that of ${\mathcal R}_{K,P}$ to study the $p$-rationality of some totally 
real number fields; he studies the case of quadratic fields with general Fibonacci sequences 
(from the fundamental unit), a method that will be rediscovered by some authors
to characterize the $p$-rationality.

\smallskip
Mention the relative version of the Coates formula in the
totally real case:

\begin{theorem} \cite[Th\'eor\`eme III.3]{Gr01}\,[Gr1982]. Let $L/K$
be an abelian extension of totally real number fields fulfilling
the Leopoldt conjecture. Let ${\mathcal N}_{L/K}$ be the group of 
local norms and let $\Cl_{L/K}^{\rm\, gen} := {\rm Gal}(H_L^{\rm ab}/LH_K)$ 
be the $p$-genus group in $L/K$; the superscript~${}^*$ denotes  
${\rm Ker}({\rm N}_{L/K})$. Then:
\begin{equation*}
\begin{aligned}
\order {\mathcal T}_{L,P}\  \sim \  \frac{\order {\mathcal T}_{K,P}}
{ \big[L \,\cap\, H_{K,P} : L \,\cap\, K^{\rm c} \big ]} \times 
 \frac{\prod_{{\mathfrak l} \, \nmid \, p} e_{{\mathfrak l},p}}{[L : L \,\cap\, H_{K,P}]} 
& \times \order \Cl_{L/K}^{\rm \, gen} 
\times \big (E_K \,\cap\, {\mathcal N}_{L/K} : {\rm N}_{L/K}(E_L) \big) \\
& \times \big ({\rm log}_P(U_{L,P}^*) : {\rm log}_P (\ov E_L^{\,*}) \big)
\times \big({\rm tor}_{\Z_p}(U_{L,P}^*) : \mu_p^* \big), 
\end{aligned}
\end{equation*}

\noindent
where $\mu_p^*=1$ for $p\ne 2$ and $\order \mu_2^* = {\rm gcd}(2,[L:K])$.
\end{theorem}

\subsubsection{Cohomological interpretation} \label{1bis}
\smallskip
In \cite{Ng2}\,[Nqd1986], Nguyen Quang Do gives the cohomological interpretation
of the dual of ${\mathcal T}_{K,P}$:
${\mathcal T}_{K,P}^* \simeq {\rm H}^2({\mathcal G}_{K,P},\Z_p)$,
considered as the first of the mysterious non positive twists 
${\rm H}^2 ({\mathcal G}_{K,P}, \Z_p(i))$ of the motivic cohomology;
for concrete results of genus type about the corresponding case of motivic 
tame kernels, see Assim--Movahhedi \cite{AM}\,[AM2019] and its important bibliography 
which would deserve to be in part among our references, despite it is beyond our goals.

\smallskip
It is indeed well known that ${\rm H}^2({\mathcal G}_{K,P},\Z_p)$ 
does appear as a tricky obstruction in many questions of Galois theory 
over number fields, whatever the technical approach. For 
${\rm H}^2({\mathcal G}_{K,S},\Z_p)$, see \cite[Appendix]{Gr3}[Gr2003].

\smallskip
But considering the two ``equivalent'' invariants 
${\rm H}^2({\mathcal G}_{K,P}, \Z_p)$ and ${\mathcal T}_{K,P}$, only the last 
one may be used, with arithmetic or analytic tools, to obtain nume\-rical experiments
and to understand the true intrinsic $p$-adic difficulties.

\subsubsection{Principal Conjectures and Theorems} \label{2}
\smallskip
Considering the invariants $\Cl_K$ and ${\mathcal T}_{K,P}$
as fundamental objects, we have given, for the abelian fields $K$, the conjectural 
behaviour of their isotopic $\chi$-components for irreducible
$p$-adic characters $\chi$ in \cite{Gr00}\,[Gr1977]; the proofs of these conjectures and
of some improvements in Iwasawa's theory are well known and the reader
may refer to the illuminating paper of Ribet \cite{Ri}\,[Rib2008] about the so-called
``Principal Theorem'' stemming from Bernoulli--Kummer--Herbrand
then Ribet--Mazur--Wiles--Thaine--Rubin--Kolyvagin--Greither
works on cyclotomy and $p$-adic $L$-functions, as a prelude of 
wide generalizations in the same spirit.

\subsection{Basic $p$-adic properties of  ${\mathcal A}_{K,P}$  $\&$ 
${\mathcal T}_{K,P}$} \label{3}
\smallskip
During the 1980's, we have written in \cite{Gr01,Gr02,Gr03}\,[Gr1982-1983-1984] 
the main properties of the groups ${\mathcal T}_{K,P}$ with their behaviour in any 
extension $L/K$ and proved (assuming Leopoldt's conjecture
 in the Galois closure of $L$) that the transfer maps:
$${\mathcal A}_{K,P} \too {\mathcal A}_{L,P} \ \ \& \ \  {\mathcal T}_{K,P} 
\too {\mathcal T}_{L,P}$$ 

\noindent
are always injective \cite[Th\'eor\`eme I.1]{Gr01}\,[Gr1982]; which has major consequences 
for the arithmetic of number fields (e.g., non-capitulation in an extension contrary 
to class groups). Of course, this property has been obtained soon after by Jaulent, 
Nguyen Quang Do and others with different techniques.

\subsubsection{The $p$-adic ${\rm Log}_S^{}$-functions} 
\smallskip
\begin{definition} \cite[\S\,2, Th\'eor\`eme 2.1]{Gr02}\,[Gr1983],
\cite[\S\,III.2.2]{Gr3}\,[Gr2003].\label{4}
Let  $I_{K,P}$ be the group of prime to $p$ ideals  of $K$.
We define the logarithm function:
$${\rm Log}_P^{} : I_{K,P} \too \Big(
\plus_{{\mathfrak p} \in P} K_{\mathfrak p}\Big) \Big /\Q_p {\rm log}_P(E_K)$$ 

\noindent
as follows. For any ideal ${\mathfrak a} \in I_{K,P}$ let $m$ be such that
${\mathfrak a}^m =: (\alpha)$, $\alpha \in K^\times$, then 
${\rm Log}_P^{}({\mathfrak a}) := \frac{1}{m}  {\rm log}_P(\alpha) 
\pmod{\Q_p {\rm log}_P(E_K)}$.
\end{definition}

The main property of ${\rm Log}_P^{}$ is that for any 
ideal ${\mathfrak a} \in I_{K,P}$, ${\rm Log}_P^{}({\mathfrak a})$ defines the Artin 
symbol in the compositum $\wt {K\,}^P$ of the $\Z_p$-extensions of $K$ by means of
the canonical exact sequence:
$$1 \to {\mathcal T}_{K,P} \too {\mathcal A}_{K,P}  \mathop {\tooo}^{{\rm Log}_P^{}} 
{\rm Log}_P^{}(I_{K,P}) \simeq {\rm Gal}(\wt {K\,}^P\!\!/K) \to 1 ,$$

\noindent
which may be generalized with arbitrary $S \subseteq P$:
$$1 \to {\mathcal T}_{K,S} \too {\mathcal A}_{K,S}   \mathop {\tooo}^{{\rm Log}_S^{}} 
{\rm Log}_S^{} (I_{K,S}) \simeq {\rm Gal}(\wt {K\,}^S\!\!/K) \to 1 ,$$

\noindent
with an obvious definition of ${\rm Log}_S^{}({\mathfrak a})$ in
$\plus_{{\mathfrak p} \in S} K_{\mathfrak p}$ modulo $\Q_p{\rm log}_S^{}(E_K)$.

\smallskip
This formalism is equivalent to that given by the theory of pro-$p$-groups 
(here ${\mathcal G}_{K,P}$), but may yield numerical computations as follows:

\smallskip
The formula for $\order {\mathcal T}_{K,S}$, $S \subseteq P$,
is the following \cite[Theorems III.2.5]{Gr1}\,[Gr1986], \cite[Corollary III.2.6.1]{Gr3}\,[Gr2003]
(under Leopoldt's conjecture):
\begin{equation}\label{TLog}
\order {\mathcal T}_{K,S} = \order {\mathcal W}_{K,S} \times
\order {\mathcal R}_{K,S} \times
\frac {\order \Cl_K}{\big ({\Z_p}{\rm Log}_S^{}(I_{K,S}) :
{\Z_p}{\rm Log}_S^{} (P_{K,S}) \big)}, 
\end{equation}

\noindent
where $P_{K,S}$ is the group of principal ideals prime to $S$,
so that ${\Z_p}{\rm Log}_S^{} (P_{K,S})$ depends obviously on
${\rm log}_S(U_{K,S})$ modulo $\Q_p {\rm log}_S(E_K)$. 
When $S \subsetneq P$, ${\mathcal W}_{K,S}$ is not necessarily equal
to ${\rm tor}_{\Z_p}(U_{K,S})/ \iota_S^{}(\mu_K)$ (cf. Lemmas \ref{exact}, \ref{leo}). 

The denominator in \eqref{TLog} gives the degree $[\wt {K\,}^S\!\! \!\cap \! H_K : K]$
and the quotient gives $\order \,\wt {\Cl_K\!}^{\!S}$.
 
\smallskip
For $S=P$, the ${\rm Log}_P^{}$-function allows, when $\mu_p \subset K$, the numerical
determination of the initial Kummer radical contained in $\wt {K\,}^P$ 
\cite{Gr0}\,[Gr1985], \cite{Ja1}\,[Jau1986].

\subsubsection{Fixed point formula}
\smallskip
Then we have obtained a fixed point formula for $S=P$ which, contrary
to Chevalley's formula for class groups in cyclic extensions
\cite{Che}\,[Che1933], does exist whatever the Galois extension $L/K$
(\cite[\S\,5]{Gr02}\,[Gr1983], \cite[Section~2\,(c)]{Ja0}\,[Jau1984], 
\cite[Proposition 6]{Gr1}\,[Gr1986], \cite[Appendice I]{Mo1}\,[Mov1988], 
\cite[Appendice]{MN}\,[MN1990]):

\begin{theorem}\label{fix}\cite[\S\,IV.3, Theorem 3.3]{Gr3}\,[Gr2003].
Let $L/K$ be a Galois extension of number fields and
$G := {\rm Gal}(L/K)$. Let~$p$ be a prime number;
we assume that $L$ satisfies the Leopoldt conjecture for~$p$. Then:
$$\order{\mathcal T}_{L,P}^G  = \order {\mathcal T}_{K,P} \times
\frac{\prd_{{\mathfrak l}\,\notdiv\, p} e^{}_{{\mathfrak l},p}}
{\Big( \sm_{{\mathfrak l}\,\notdiv\, p}\,
\hbox{$\frac{1}{e^{}_{{\mathfrak l},p}}$}  \Z_p {\rm Log}_P^{}({\mathfrak l}) + 
\Z_p {\rm Log}_P^{}(I_{K,P}): \Z_p {\rm Log}_P^{}(I_{K,P}) \Big)}, $$

\noindent
where $e^{}_{{\mathfrak l},p}$ is the $p$-part of the ramification index of 
${\mathfrak l}$ in $L/K$.
\end{theorem}

\begin{remark}
Contrary to the computation of 
${\rm tor}_{\Z_p}^{}(U_{K,P} / \ov {E\,}^P_{\!\!K})$,
that of the $\Q_p$-vector space $\Q_p {\rm log}_P(E_K)$ {\it does
not need the knowledge of the group of units $E_K$}; it only depends
of Leopoldt's conjecture (assumed) and its $\Q_p$-dimension is
$r_1+r_2-1$; the case of $\Q_p {\rm log}_S(E_K)$ is more mysterious.
\end{remark}

The case of totally real fields is easier since the ${\rm Log}$-function
trivializes because we have $\bigoplus_{{\mathfrak p} \in P} K_{\mathfrak p} =
\Q_p {\rm log}_P(E_K)\, \plus\, \Q_p$, which allows explicit computations
\cite[Th\'eo\-r\`eme III.1]{Gr01}\,[Gr1982]:

\begin{corollary}\label{fixreal}\cite[Exercise IV.3.3.1]{Gr3}\,[Gr2003].
In the case of a totally real number field $L$, the above formula becomes
(under Leopoldt's conjecture):
$\order{\mathcal T}_{L,P}^G = \order {\mathcal T}_{K,P} \cdot 
p^{\rho - r} \cdot \prod_{{\mathfrak l} \nmid p} e^{}_{{\mathfrak l},p}$,
where $p^r \sim [L : K]$ and $\rho$ only depends on the 
decomposition of the ramified primes $\ell \nmid p$ in $L/K$.
\end{corollary}

\subsubsection{$p$-primitive ramification} \label{5}
\smallskip
The fixed point formula of Theorem \ref{fix} allows to characterize the case where
$\order {\mathcal T}_{L,P}=1$ in a $p$-extension $L/K$:

\begin{corollary}\label{TL=1} Let $L/K$ be any finite $p$-extension.
Then ${\mathcal T}_{L,P}=1$ if and only if the two following conditions 
are fulfilled (under Leopoldt's conjecture):

\smallskip
\quad (i) ${\mathcal T}_{K,P} =1$;

\quad (ii) $\Big( \sm_{{\mathfrak l}\,\notdiv\, p}\,
\hbox{$\frac{1}{e^{}_{{\mathfrak l},p}}$}  \Z_p {\rm Log}_P^{}({\mathfrak l}) + 
\Z_p {\rm Log}_P^{}(I_{K,P}): \Z_p {\rm Log}_P^{}(I_{K,P}) \Big) =
\prd_{{\mathfrak l}\,\notdiv\, p} e^{}_{{\mathfrak l},p}$.
\end{corollary}

\begin{definition} \cite[\S\,IV.3, (b)]{Gr3}\,[Gr2003].
When the condition (ii) is fulfilled, we say
that the $p$-extension $L/K$ is $p$-primitively ramified 
and that the set $T$ of tame places ${\mathfrak l}$, ramified
in $L/K$, is primitive \cite[Ch.\,III, Definition \& Remark]{Gr1}\,[Gr1986], 
which is equivalent (in terms of Frobenius automorphisms) to:
\begin{equation}
{\rm Gal} (\wt {K\,}^P \!\!/K) \simeq {\mathcal A}_{K,P}/{\mathcal T}_{K,P}
= \plus_{{\mathfrak l} \in T} \big\langle 
\big(\hbox{$\frac{\wt {K\,}^P \!\! /K}{{\mathfrak l}}$} \big) \big\rangle.
\end{equation}
\end{definition}

Of course, any $P$-ramified extension is $p$-primitively ramified.

\smallskip
Then in \cite[Ch. III, \S\,2, Theorem 2  \& Corollary]{Gr1}\,[Gr1986] are characterized,
for $p=2$ and $p=3$, the abelian $p$-extensions $K$ 
of $\Q$ such that ${\mathcal T}_{K,P} =1$.
This is connected with the ``regular kernel'' of $K$ which, from results of Tate,
follows similar properties which have been explained in a joint work with
Jaulent \cite{GJ}\,[GJ1989] and developped in Jaulent--Nguyen Quang Do
\cite{JN}\,[JN1993]. We can state:

\begin{theorem} \cite[Theorem III.4.2.5, Theorem IV.3.5]{Gr3}\,[Gr2003].\label{triviality}
Let $K$ be any number field.The following properties are equivalent:
 
\smallskip
\quad (i) $K$ satisfies the Leopoldt conjecture at $p$ and
${\mathcal T}_{K,P}=1$;

\smallskip
\quad (ii) ${\mathcal A}_{K,P} := {\mathcal G}_{K,P}^{\rm ab} 
= {\rm Gal}(H_{K,P} / K) \simeq \Z_p^{r_2+1}$,

\smallskip
\quad (iii) the Galois group ${\mathcal G}_{K,P}$ is a free pro-$p$-group on 
$r_2+1$ generators, which is equivalent to fulfill the following four conditions:

\smallskip
\qquad $\bullet$ $K$ satisfies the Leopoldt conjecture at $p$,

\smallskip
\qquad $\bullet$ $\Cl_K \simeq \Z_p{\rm Log}_P^{}(I_{K,P})\big / 
\big( {\rm log}_P(U_{K,P}) + \Q_p {\rm log}_P^{}(E_K) \big)$,

\smallskip
\qquad $\bullet$ ${\rm tor}_{\Z_p}(U_{K,P}) = \mu_p(K)$, 

\smallskip
\qquad $\bullet$ $\Z_p{\rm log}_P^{}(E_K) \hbox{ is a direct summand in } 
{\rm log}_P^{}(U_{K,P})$.
\end{theorem}

\subsection{New formalisms and use of pro-$p$-group theory}\label{new}
\smallskip
\subsubsection{Infinitesimal arithmetic} \label{6}
\smallskip
From \cite{Ja0,Ja1,Ja11,Ja2}\,[Jau1984-1986-1994-1998].
At the same time, in his Thesis, Jaulent defines the {\it infinitesimal
arithmetic} in a number field proving, in a nice conceptual framework, 
generalizations of our previous results, especially in the new context of
{\it logarithmic classes} \cite{Ja11,Ja3}\,[Jau1994-2002], adding Iwasawa theory results,
study of the $p$-regularity (replacing ${\mathcal T}_{K,P}$ by
the tame kernel ${\rm K}_2(Z_K)$ of the ring of integers of $K$), 
and genus theory. 

\smallskip
The same technical context of $\ell\,(=p)$-adic class 
field theory and a logarithmic class field theory was developed later in 
much papers, including computational
methods of Bourbon--Jaulent \cite{BJ}\,[BJ2013]. 
He studies in \cite{Ja11}\,[Jau1994]
 the logarithmic class group $\wt \Cl_K$ (do not confuse with $\wt {\Cl_K}^P$)
whose finitness is equivalent to the Gross (or Gross--Kuz'min) conjecture 
\cite{FGS}\,[FGS1981], \cite{Kuz}\,[Kuz1972] (a survey is given in \cite[\S\,III.7]{Gr3}\,[Gr2003]);
see also some comments in \cite{Ja10,Ja35}\,[Jau1987-2017].

Some properties of capitulation of generalized ray class groups and of
$\wt \Cl_K$ are given in \cite{Ja33,Ja34,Ja55,Ja44}\,[Jau2016a-2016b-2018a-2018b].

\subsubsection{Pro-$p$-group theory version} \label{7}
\smallskip
Shortly after, at the end of the 1980's, in his thesis, Movahhedi
\cite{Mo1,Mo2}\,[Mov1988-1990] gives a wide study of the abelian $p$-ramification theory, 
using mainely the properties of the pro-$p$-group ${\mathcal G}_{K,S}$ and deduces again
most of the previous items, then he gives the main structural and cohomological 
properties of ${\mathcal G}_{K,P}$ and the classical characterization of the
triviality of ${\mathcal T}_{K,P}$.
He proposes for this to speak of ``$p$-rational fields'' \cite[Definitionn~1]{Mo2}\,[Mov1990], 
that is to say the number fields $K$ such that Leopoldt's conjecture holds for $p$ 
and ${\mathcal T}_{K,P} =1$ (cf. Theorem \ref{triviality}); this was
inspired by the fact that $\Q$ is (obviously) $p$-rational for all $p$. 
This vocabulary has been adopted by the arithmeticians. 

\smallskip
Then Movahhedi gives properties of $p$-rational extensions $L/K$ and
the reciprocal of our result characterizing the $p$-rationality in a $p$-extension 
$L/K$, in other words the ``going up'' of the $p$-rationality:

\begin{theorem}\cite[Th\'eor\`eme 3, \S\,3]{Mo1}\,[Mov1988]. Let $L/K$ be a $p$-extension
of number fields. The field $L$ is $p$-rational if and only if $K$ is
$p$-rational and the set $T$ of tame primes, ramified in $L/K$, is
$p$-primitive in $K$. Moreover, under these conditions, the extension 
$T(L)$ of $T$ to $L$ is $p$-primitive.
\end{theorem}

This implies that if $K$ is $p$-rational and $T$ $p$-primitive, 
then any $T$-ramified $p$-extension $L/K$ fulfills the Leopoldt conjecture
and $T(L)$ is $p$-primitive (a particular case was given in
\cite[Th\'eor\`eme III.4]{Gr01}\,[Gr1982] for totally real fields).

\begin{remark} In practice, in research papers, one assumes in general 
an universal Leopoldt conjecture, so that the above statement becomes:

\centerline{\it $L$ is $p$-rational if and only if $K$ is $p$-rational and $T$
is $p$-primitive} 

\smallskip
\noindent
(equivalent to use the fixed point formula of Theorem \ref{fix} and Corollary \ref{TL=1}).
\end{remark}

In the 1990's, the classical results on $p$-ramification, $p$-rationality,
and $p$-regularity about the triviality of the tame kernel ${\rm K}_2 (Z_K)$,
are amply illustrated in various directions by Movahhedi, Nguyen Quang Do, 
Jaulent (see Movahhedi \cite{Mo2}\,[Mov1990], 
Movahhedi--Nguyen Quang Do {\cite{MN}\,[MN1990], Berger--Gras \cite{BG}\,[BG1992],
Jaulent--Nguyen Quang Do \cite{JN}\,[JN1993], Jaulent--Sauzet \cite{JS1}\,[JS1997],
and Jaulent \cite{Ja2}\,[Jau1998]): pro-$p$-group theory with explicit 
determination of a system of generators and relations for ${\mathcal G}_{K,S}$,
Galois cohomology, Iwasawa's theory, Leopoldt and Gross conjectures. 

\smallskip
Recall that in \cite[Scolie, p.\,112]{Ja10}\,[Jau1987] Jaulent shows that, when $\mu_p \subset K$ 
the nullity of the $p$-Hilbert kernel ${\rm H}_2(L) \otimes \Z_p$ implies Leopoldt and 
Gross conjectures. 
Moreover \cite{Ja2}\,[Jau1998] deals with ramification and decomposition.

\smallskip
Under the assumptions: $\mu_p \subset K$, ${\rm H}_2(L) \otimes \Z_p=0$, for the
Hilbert kernel, and the existence of ${\mathfrak p}_0 \in S$ such that 
$\mu_{K_{{\mathfrak p}_0}} = \mu_K$, some results in \cite{Ng3}\,[Nqd1991],
after \cite{Mo1}\,[Mov1988] and \cite{MN}\,[MN1990] 
on the primitive reciprocity laws, in the framework of $p$-rationality,
describe (by means of generators and relations) the Galois group 
${\mathcal G}_{K,S}$.

\subsubsection{Links between these invariants and Iwasawa's theory}\label{iwa}
\smallskip
Despite the fact that we limit ourselves to arithmetical invariants of the base field
(which is always possible), we give a short overview on the Iwasawa context
and we indicate the main references for the reader.\footnote{\,This Subsection,
describing the two different (but equivalent) techniques, is 
close to personal communications of Jean-Fran\c cois Jaulent and Thong Nguyen Quang Do 
(up to the notations and some comments). We thank them also for some remarks and corrections
about this subsection.}

\smallskip
The base field invariants concerned are (in the case $S=P$), the torsion group
${\mathcal T}_{K,P}$, the $p$-Hilbert kernel ${\rm H}_2(K) \otimes \Z_p$, and
the logarithmic class group $\wt \Cl_K$. 

\smallskip
Let $K_\infty := K(\mu_{p^\infty})$, $\Gamma := {\rm Gal}(K_\infty/K) =: 
\langle \gamma \rangle$, $X$ the Galois group of the maximal abelian
pro-$p$-extension of $K_\infty$, non-ramified and in which all places totally split.
For a field $k$, we put $\mu_{p^\infty}(k) = \mu_{p^\infty} \,\cap\, k^\times$.
For any module $M$ over the Iwasawa algebra, denote by $M(i)$ the $i$th twist
on which $\Gamma$ acts by $\gamma  \cdot m := \kappa^i(\gamma) \cdot m^\gamma$,
where $\kappa$ is the cyclotomic character.

\smallskip
Then the interpretation of the above invariants, in the Iwasawa framework 
is given, in part, by the following two results:

\begin{theorem}  \cite[Theorem 4.2]{Ng2}\,[Nqd1986]. Assuming the Leopoldt conjecture for $p$ in $K$, 
one has the following exact sequence
$1 \to \mu_{p^\infty}(K) \too \plus_{{\mathfrak p} \mid p} \mu_{p^\infty}(K_{\mathfrak p})
\too {\mathcal T}_{K,P} \too {\rm Hom}_\Gamma(X, \mu_{p^\infty}) \to 1$.
Then we have the following relation:
$${\rm Hom}_\Gamma(X, \mu_{p^\infty}) = {\rm Hom}_\Gamma (X(-1), \Q_p/\Z_p) 
 = {\rm Hom} (X(-1)_\Gamma, \Q_p/\Z_p) \simeq {\rm Gal} \big (H_K^{\rm bp} / \wt {K\,}^P \big)$$
 
\noindent
(see Remark \ref{rema1}), while (in relation with the paper of Federer--Gross--Sinnot
\cite{FGS}\,[FGS1981]):
\begin{equation}\label{X=C}
X_\Gamma \simeq  \wt \Cl_K.
\end{equation}
\end{theorem}

The relation \eqref{X=C} is given in \cite{Ja1}\,[Jau1986], then in \cite{Ja100,Ja2}\,[Jau1990-1998].

\smallskip
The considerable advantage of $\wt \Cl_K$, introduced in \cite{Ja11}\,[Jau1994],
is that it only involves some specific and explicit notions of classes and units of the 
base field $K$ and is then likely to be numerically calculated 
(Belabas--Jaulent  \cite{BJ0}\,[BJ2016]).

\smallskip
When $i$ varies, similar results may be interpreted by means of 
higher ${\rm K}$-groups \cite{Ng33}\,[Nqd1992].
The main ${\rm K}_2$-theoretic interpretation is given as follows:

\begin{theorem}\cite[Theorem 5.6]{Ng2}\,[Nqd1986]. One has:
$({\rm H}_2(K)\otimes \Z_p)^* = {\rm Ker}^2_P(\Q_p/\Z_p(-1))$;
if $K$ contains $\mu_{p^e}$, $e \geq 1$, one obtains the perfect duality:
${\rm Gal} \big (H_K^{\rm bp} / \wt {K\,}^P \big)[p^e] \times 
\big ({\rm H}_2(K)/p^e {\rm H}_2(K) \big )(-1) \too \mu_{p^e}$, 
where $T[p^e] := \{x \in T, \ \, p^e \cdot x =0\}$ for a $\Z_p$-module $T$, and where
${\rm Ker}^2_P$ is the kernel of the localization homomorphism
${\rm H}_2({\mathcal G}_{K,P}) \too \plus_{{\mathfrak p} \mid p} 
{\rm H}_2({\mathcal G}_{K,{\mathfrak p}})$.
\end{theorem}

This result of duality does appear in \cite{Ja1,Ja100}\,[Jau1986-1990].
If $\mu_p \subset K$, the nullity of ${\rm H}_2(K)\otimes \Z_p$ is equivalent to that of
$\wt \Cl_K$, which makes the link with the above Scolie 
\cite[Scolie, p.\,112]{Ja10}\,[Jau1987] of Jaulent.
For relations between logarithmic classes and higher ${\rm K}$-groups, mention 
the work of Jaulent--Michel \cite{JM}\,[JM2006] and that of Hutchinson \cite{Hut}\,[Hut2017].

\subsubsection{Synthesis $2003$--$2005$} \label{8}
\smallskip
Because our Crelle papers, were written in french, whence
largely ignored, all the results and consequences, that we have given in 
\cite{Gr00,Gr01,Gr02,Gr03,Gr0,Gr1,Gr2}\,[Gr1977-1982-1983-1984-1985-1986-1998], 
were widely developed and improved
in \cite{Gr3}\,[Gr2003] where a systematic and general use of ramification 
and decomposition is considered, the infinite places playing 
a specific role (decomposition or complexification). 

\smallskip
Furthermore, \cite[Theorem V.2.4 and Corollary V.2.4.2]{Gr3}\,[Gr2003]
give a characterisation (with explicit governing fields) of the existence 
of degree $p$ cyclic extensions of $K$ with given ramification and 
decomposition. This criteria has been used by Hajir--Maire
and Hajir--Maire--Ramakrishna in several of their papers for results 
on $S$-ramified pro-$p$-groups (see, e.g., \cite[Theorem 5.3]{HMR1}\,[HMR2019a],
\cite[Remark 2.2.]{HMR2}\,[HMR2019b] for the most recent publications).

\subsection{Present theoretical and algorithmic aspects} 
\smallskip
One may say that there is no important progress for $p$-rationality,
for itself, since $p$-rational fields are in some sense the ``simplest fields''
in a $p$-adic sense, 
but that the significance of the $p$-adic properties of the groups 
${\mathcal T}_{K,S}$, in much domains of number theory, has given a great lot
of heuristics, conjectures, computations; so we shall now describe
some of these aspects with some illustrations (it is not possible to be comprehensive
since the concerned literature becomes enormous).

\subsubsection{Absolute abelian Galois group $A_K$ of $K$} \label{9}  
\smallskip
Let $K^{\rm ab}$ be the maximal abelian pro-extension of $K$.
In \cite{AS}\,[AS2013] Angelakis--Stevenhagen, after some 
work by Kubota \cite{Kub}\,[Kub1957] and Onabe \cite{On}\,[Ona1976], provide a direct 
computation of the {\it profinite group} $A_K := {\rm Gal}(K^{\rm ab}/K)$ for imaginary quadratic 
fields $K$, and use it to obtain many 
different $K$ that all have the same ``minimal'' absolute abelian Galois group, which
is in some sense a condition of minimality of the groups ${\mathcal T}_{K,P}$
for all primes $p$. They obtain for instance, among other results and numerical 
illustrations:

\begin{theorem}  \cite[Theorem 4.1 $\&$ Section 7]{AS}\,[AS2013].
An imaginary quadratic field $K \ne \Q(i)$, 
$\Q(\sqrt{-2})$ of class number $1$ has absolute abelian Galois group 
isomorphic to $\widehat \Z^2 \times \prod_{n \geq 1}\Z/n\Z$.
\end{theorem}

This corresponds to the fact that such fields are $p$-rational for all $p$
(up to the factors ${\mathcal W}_{K,P}$ for $p=2,3$).
Then the generalization to an arbitrary $K$ involves
the ${\mathcal T}_{K,P}$ for all primes $p$, giving:

\begin{theorem} \cite[Theorem 2.1  $\&$ Corollary 2.1]{Gr4}\,[Gr2014].
Let  $K^{\rm ab}$ be the maximal Abelian pro-extension of $K$. 
Let ${\mathcal H}_K$ be the compositum, over $p$, of the maximal
$P$-ramified Abelian pro-$p$-extensions $H_{K,P}$ of $K$.
Under the Leopoldt conjecture, there exists an Abelian extension $L_K$ of $K$
such that ${\rm Gal}(L_K/K) \,\simeq\,\prod_p {\mathcal T}_{K,P}$  and 
such that ${\mathcal H}_K$ is the direct compositum over $K$ of $L_K$ 
and the maximal $\widehat \Z$-extension of $K$, and such that
we have the non-canonical isomorphism (for some explicit integers $\delta$ and $w$):
$${\rm Gal} \big (K^{\rm ab}/L_K \big )  =  \widehat \Z^{r_2+1} \! \times\!
{\rm Gal} \big (K^{\rm ab} /{\mathcal H}_K \big ) 
\simeq \widehat \Z^{r_2+1} \! \times\! 
\prd_{n \geq 1}\!\! \Big((\Z/2\,\Z)^\delta \!\times\! \Z/w \, n \,\Z\Big).$$
\end{theorem}

Angelakis--Stevenhagen conjecture in \cite[Conjecture 7.1]{AS}\,[AS2013]
the infiniteness of imaginary quadratic fields $K$ such that 
$A_K \simeq \widehat \Z^2 \times \prod_{n \geq 1}\Z/n\Z$. 

\smallskip
Note that when the $p$-class group of $K$ is non-trivial, $K$ is $p$-rational
if and only if $\Cl_K$ is cyclic and the $p$-Hilbert class field $H_K$ is contained
in $\wt {K\,}^P$ (assuming ${\mathcal W}_{K,P}=1$).

\smallskip
Whence the importance of fields $K$ being $p$-rational for all $p$
(or more precisely such that ${\mathcal T}_{K,P} = {\mathcal W}_{K,P}$
for all $p$); it is an easier 
problem only for $\Q$ and imaginary quadratic fields,
but dreadfully difficult when $K$ contains units of infinite order since
it is an analogous question as for Fermat's quotients of algebraic 
numbers (various heuristics and conjectures in \cite{Gr5}\,[Gr2016)), or values 
of $L$-functions which intervenne as in Coates--Li
\cite{CL,CL2}\,[CL2018-2019], Goren \cite{Go}\,[Gor2001], and more or less, 
in many papers as Boeckle--Guiraud--Kalyanswamy--Khare
\cite{BGKK}\,[BGKK2018] when the normalized $p$-adic regulator is a unit.
We have conjectured that, in any given number field $K$,
${\mathcal T}_{K,P} = 1$ for almost all $p$.

\subsubsection{Greenberg's conjecture on Iwasawa's $\lambda$, $\mu$} 
\smallskip
For a totally real number field $K$, consider (under the Leopoldt conjecture)
the cyclotomic $\Z_p$-extension $K^{\rm c}$ of $K$. Then Greenberg has
conjectured in \cite{Gre1}\,[Gre1976]
that the Iwasawa's invariants $\lambda$ and $\mu$ are zero. 

\smallskip
Equivalent formulations of this conjecture have been given,
as in \cite{Ng6}\,[Nqd2018] for an encompassing approach covering the necessary 
and sufficient conditions considered by Greenberg in two particular 
cases (we give up to provide a complete bibliography), but we must
mention that the two invariants ${\mathcal T}_{K,P}$ and
$\wt \Cl_K$ (the logarithmic class  group of Jaulent) are in some
sense ``governing invariants'' for the Greenberg conjecture 
(in a theoretical and numerical viewpoint)
and explain the $p$-adic obstuctions for a standard proof in the framework of 
Iwasawa's theory; for instance, as soon as ${\mathcal T}_{K,P}=1$ or 
$\wt {\Cl}_K=1$, Greenberg's conjecture is true for trivial reasons. For this, see 
\cite[Th\'eor\`emes 3.4, 4.8, 6.3]{Gr6}\,[Gr2017a] and \cite{Gr61}\,[Gr2018b]
about ${\mathcal T}_{K,P}$,
then the interpretation by Jaulent with the group of universal norms \cite{Ja6}\,[Jau2019b]
and the following criterion (under the Gross-Leopoldt conjecture):

\begin{theorem}\cite[Th\'eor\`eme 7, \S\,1.4]{Ja4}\,[Jau2018a]. The totally real
number field $K$ fulfills the conjecture of Greenberg if and only if its 
logarithmic class group $\wt \Cl_K$ capitulates in the cyclotomic $\Z_p$-extension 
$K^{\rm c}$ of $K$.
\end{theorem}

If Greenberg's conjecture is true (which is no doubt), such 
general condition of capitulation is very reassuring since
we recall that, on the contrary, the group ${\mathcal T}_{K,P}$ 
{\it never capitulates}. Moreover the property of capitulation
(well known in Hilbert's class fields) is more general 
for generalized ray class groups and, especially,
is possible in {\it absolute abelian extensions} as shown in many 
papers including \cite{Gr111}\,[Gr1997], Bosca \cite{Bos}\,[Bos2009],
then \cite{Ja55,Ja44}\,[Jau2018b-2019a]. 

\smallskip
This result may be deduced in the framework of Iwasawa's theory
recalled in the \S\,\ref{iwa} \cite[Th\'eor\`eme 2.1]{Ng6}\,[Nqd2018].

\smallskip
Unfortunately, at the time of writing this text, no proof of Greenberg's 
conjecture does exist, despite some unsuccessful attempts in
\cite{Ng4,Ng5}\,[Nqd2006-2017] (to understand the key-points of the $p$-adic 
obstruction to be analyzed and possibly completed, 
see \cite[\S\,3.4, Remark]{Ja4}\,[Jau2018a] and \cite[\S\,6.2, Diagram]{Gr61}\,[Gr2018b]).

\subsubsection{Galois representations with open image} \label{10}
\smallskip
For constructions by Greenberg, in \cite{Gre2}\,[Gre2016], of continuous Galois representations 
${\rm Gal}(\ov \Q/\Q) \to {\rm GL}_n(\Z_p)$ with open image, 
the $p$-rational fields play a great role, and the first obvious case is
that of $p$-regular cyclotomic fields $\Q(\mu_p)$ which are $p$-rational
(yet reported by \cite{Sa}\,[Sha1964], \cite{Gr1}\,[Gr1986], and generalized by 
introducing in \cite{GJ}\,[GJ1989] the notion of $p$-regularity of number fields that we do not 
develop in this paper, for short, but which behaves as $p$-rationality; 
see a survey in \cite{JN}\,[JN1993]).

\smallskip
Then, an interesting typical conjecture is the following:

\begin{conjecture}\cite[Conjecture 4.2.1]{Gre2}\,[Gre2016].
For any odd prime $p$ and for any $t \geq 1$, there exists a $p$-rational field $K$ 
such that ${\rm Gal}(K/\Q) \simeq (\Z/2 \Z)^t$. 
\end{conjecture}

Numerical examples and statistics have been given for various $p$ and $t$;
see \cite{Gre2}\,[Gre2016] and (by Robert Bradshaw) the $3$-rationality of:
$$K=\Q(\sqrt{13},\sqrt{145},\sqrt{209},\sqrt{269},\sqrt{373},\sqrt{-1}). $$

Some PARI/GP programs are given in Pitoun--Varescon \cite{Pi,PV}\,[Pit2010-PV2015], 
and \cite[\S\,5.2]{Gr8}\,[Gr2017c]
showing the $3$-rationality of:
$$K=\Q(\sqrt{-2},\sqrt{-5},\sqrt{7},\sqrt{17},\sqrt{-19},\sqrt{59}). $$

For fixed $p$ (e.g., $p=3$), the probability of $p$-rationality decreases 
dramatically when $t \to \infty$; indeed, if
${\rm Gal}(K/\Q) \simeq (\Z/2 \Z)^t$, $K$ is $p$-rational if
and only if the $2^t-1$ quadratic subfields $k$ of $K$ are $p$-rational
whose probability is of the order of $\big(\frac{1}{p}\big)^{2^t-1}$
assuming that class groups and units of each $k$ 
are random and largely independent regarding the $p$-adic properties.

\subsubsection{Order of magnitude of ${\mathcal T}_{K,P}$ and conjectures} \label{11} 
\smallskip
We have conjectured in \cite[Conjecture 8.11]{Gr5}\,[Gr2016] that for a fixed number 
field $K$, ${\mathcal T}_{K,P}=1$ for all $p \gg 0$. Moreover, all numerical calculations
show that the non-$p$-rationality constitutes an exception.

\smallskip
In another direction, fixing $p$ and taking $K$ in some given infinite family 
${\mathcal K}$ (e.g., real fields of given degree $d$) we have given extensive 
numerical computations in direction of the following 
``$p$-adic Brauer--Siegel'' property:

\begin{conjecture}\cite[Conjecture 8.1]{Gr10}\,[Gr2019a]. \label{BS}
There exists a constant 
${\mathcal C}_p({\mathcal K})$ such that:
$$v_p(\order {\mathcal T}_{K,p}) \leq {\mathcal C}_p({\mathcal K}) \cdot 
\Frac{{\rm log}_\infty(\sqrt{D_K})}{{\rm log}_\infty(p)}, $$ 

\noindent
for all $K \in {\mathcal K}$, where ${\rm log}_\infty$ is the usual complex 
logarithm.
\end{conjecture}

Thus there are two questions about $C_p(K):= 
\Frac{v_p(\order {\mathcal T}_{K,P}) \cdot {\rm log}_\infty(p)}{{\rm log}_\infty(\sqrt{D_K})}$
and the quantities ${\mathcal C}_p := \ds\sup_K(C_p(K))$, 
${\mathcal C}_K :=\ds \sup_p(C_p(K))$:

\quad (i) The existence of ${\mathcal C}_K < \infty$, for a given $K$, only says that
the Conjecture ``${\mathcal T}_{K,P}=1\ \forall p \gg 0$'' is true for the field $K$; 
for this field, we get $\ds\limsup_{p}(C_p(K)) = 0$.

\quad (ii) If ${\mathcal C}_p < \infty$ does exist for a given $p$,
we have an universal $p$-adic analog of Brauer--Siegel theorem 
(the above Conjecture \ref{BS}).

\smallskip
These questions being out of reach, many results give, on the contrary, 
the infinteness of primes $p$ yielding the $p$-rationality of a field $K$,
in general under the $a\,b\,c$ conjecture, following the method 
given by Silverman \cite{Si}\,[Sil1988], Graves--Murty \cite{GM}\,[GM2013], 
Boeckle--Guiraud--Kalyanswamy--Khare
\cite{BGKK}\,[BGKK2018], Maire--Rougnant \cite{MR2}\,[MR2019b]; for instance:

\begin{theorem}\cite[Corollary to Theorem A]{MR2}\,[MR2019b]. Let $K$ be a real quadratic 
field or an imaginary $S_3$-extension. If the generaalized 
abc-conjecture holds for $K$, then $\order\big\{p \leq x,\ \hbox{$K$ is $p$-rational} \big\} 
\geq c \cdot {\rm log}(x)$ as $x \to \infty$, for some constant $c >0$ depending on $K$.
\end{theorem}

\smallskip
This shows the awesome distance between the two aspects of the problem;
indeed, for $K=\Q(\sqrt 5)$, no prime number (up to $p < 10^{14}$ from 
Elsenhans--Jahnel: $\ $ \url{https://oeis.org/A060305}) is known giving 
${\mathcal T}_{K,P} \ne 1$.

\smallskip
In another viewpoint, as in \cite{By}\,[By2003] (after some works of Hartung, 
Horie, Naito) and \cite{AB}\,[AB2018], it is shown the infiniteness of 
$p$-rational real quadratic fields for $p=3$ and $p=5$.

\subsubsection{Fermat curves}
\smallskip
To study Fermat curves of exponent $p$, one uses the base field 
$K=\Q(\mu_p)$ and works in some Kummer extensions; for instance:

\smallskip
\quad (i) Shu \cite{Shu}\,[Shu2018] gives general formulae of the root numbers of the 
Jacobian varieties of the Fermat curves $X^p+Y^p=\delta$, where 
$\delta$ is an integer, and studies their distribution. In this article
the Vandiver conjecture or the regularity of $p$ implies some
precise properties of the Selmer groups of these Jacobian varieties.

\smallskip
\quad (ii) Davis--Pries \cite{DP}\,[DP2018] work in $P$-ramified Kummer extensions
of $K$ with $P=\{{\mathfrak p} =(1-\zeta_p)\}$, as follows. 
Let $L \subset H_{K,P}$ be defined by:
$$L = K \big(\sqrt[p]{\zeta_p}, \sqrt[p]{1-\zeta_p}, \cdots, 
\hbox{$\sqrt[p]{1-\zeta_p^r}$}\big), \ \, r=\hbox{$\frac{p-1}{2}$},$$

The Kummer radical of $L$ is also generated by the real
cyclotomic units and the numbers $\zeta_p$, $1-\zeta_p$.
In the same way as previously, non-Vandiver's conjecture or 
non-regularity for $p$ are crucial obstructions.

\smallskip
Under the {\it Vandiver conjecture}, this radical is of $p$-rank $r+1$ since it is
then given by $E_K \cdot \langle 1-\zeta_p \rangle$ modulo $K^{\times p}$.

\smallskip
Under the {\it regularity of $p$}, we get ${\mathcal T}_{K,P} = 1$
(reflection theorem \eqref{reflection}) and $L$ is the maximal $p$-elementary 
subextension of $H_{K,P}$; $L/K$ being $p$-ramified, whence $p$-primitively 
ramified (\S\,\ref{5}), this gives the $p$-rationality of $L$.

\smallskip
Let $E$ be the maximal $p$-elementary subextension of $H_{L,P}$; 
since ${\mathcal T}_{L,P} = 1$ with $E/L$ $p$-ramified, 
we then have ${\mathcal T}_{E,P} = 1$ 
and ${\rm rk}_p({\rm Gal}(E/L)) = r \cdot p^{r+1}+1$. 
One can deduce that $\Cl_L=\Cl_E=1$ since $E/K$ is totally 
ramified at ${\mathfrak p}$ (Theorem \ref{tp=1} and Chevalley's 
formula in any successive $p$-cyclic extensions in $E/K$).

\smallskip
In simple cases as $p=37$, where $\order \Cl_K=p$ and where 
$H_K \subseteq L$ in which $p$ splits, the formula of Theorem \ref{thmf}
gives ${\rm rk}_p({\mathcal T}_{L,P})= {\rm rk}_p(\Cl_L^P) + p-1$, whence
${\rm rk}_p({\rm Gal}(E/L)) = r \cdot p^{r+1}+2r+1+{\rm rk}_p(\Cl_L^P)$ 
depending on $\Cl_L^P$, a priori unknown.

\smallskip
The purpose of \cite{DP}\,[DP2018] is to get information on ${\rm H}^1({\rm Gal}(E/K),M)$
for some ${\rm Gal}(E/K)$-modules $M$, subquotients of the relative homology 
$H_1(U, Y ; \F_p)$ of the Fermat curve, where $U$ is the affine curve 
$x^p + y^p = 1$ and $Y$ the set of $2p$ cusps where $xy = 0$. They 
completely elucidate the case $p=3$.

\subsection{Computational references and numerical tables}\label{12}
\smallskip
Many references may be cited:

\smallskip
The first table for the computation of $\order {\mathcal T}_{K,P}$ for imaginary 
quadratic fields is that of Charifi \cite{Cha}\,[Cha1982], using formula \eqref{TLog}. 
In Hatada \cite{Hat1,Hat2}\,[Hat1987-1988] the computations correspond to statistics 
about the values (modulo $p$) 
of the normalized regulator ${\mathcal R}_{K,P}$ of real fields as
$K=\Q(\sqrt 5)$ by the way of Fibonacci numbers and values at 
$2-p$ of zeta-functions as we have mentioned in \S\,\ref{1}.
He obtains for instance that $\Q(\sqrt 2)$ is $p$-rational for all
$p \leq 20000$, except $p=13, 31$ (our program gives the next 
exception $p=1546463$ up to $10^8$). 

\smallskip
A precise study of $p$-rationality of imaginary quadratic fields is 
given by  Angelakis--Stevenhagen in \cite[Section 7]{AS}\,[AS2013].

\smallskip
A wide study of ${\mathcal T}_{K,P}$, with tables and publication of PARI/GP
programs, is done by Pitoun \cite[Chapitre 4]{Pi}\,[Pit2010], but these more conceptual 
programs are not so easy to manage by the reader. Then some statistical results
with tables are given by Pitoun--Varescon in \cite{PV}\,[PV20015].

\smallskip
In \cite{HZ}\,[HZ2016] Hofmann--Zhang compute the valuation of the 
(usual) $p$-adic regulators
of cyclic cubic fields with discriminant up to $10^{16}$, for $3 \leq p \leq 100$, 
and observe the distribution of these valuations.

\smallskip
About the conjecture of Greenberg \cite{Gre2}\,[Gre2016]
Kraft--Schoof \cite{KS}\,[KS1995] have computed such Iwasawa's invariants and confirm 
the conjecture for $p=3$ and conductors $f$ of real quadratic fields 
$f \not\equiv 1 \pmod 3$ up to $10^4$. 
In \cite{Gr8}\,[Gr2017c], some heuristics on the conjecture and numerical examples 
are given with programs; then we illustrate the following conjecture of Hajir--Maire
\cite[Conjecture 4.16]{HM1}\,[HM2018b]:

\smallskip
{\it Given a prime $p$ and an integer $m \geq 1$, coprime to $p$, there exist a
totally imaginary field $K_0$ and a degree $m$ cyclic extension 
$K/K_0$ such that $K$ is $p$-rational; it is conjectured that the 
statement is true taking for $K_0$ an
imaginary $p$-rational quadratic field.}

\smallskip
In \cite[Table 1, \S\,2]{BR}\,[BR2017], Barbulescu--Ray give explicit $p$-rational large 
compositum of quadratic fields. We may cite some works by Bouazzaoui \cite{Bou}\,[Bou2018], 
El Habibi--Ziane \cite{ElHZ}\,[ElHZ2018] based on $p$-rationality of quadratic fields.

\smallskip
In the similar context of $p$-ramification, a new  PARI/GP pakage allows the 
computation of the logarithmic class group $\wt \Cl_K$ of a number field by 
Belabas--Jaulent \cite{BJ0}\,[BJ2016] that we can illustrate as follows
where the invariants ${\sf [X, Y, Z]}$ are linked by the exact sequence:
$$1 \to X \too Y := \wt \Cl_K \too Z := \Cl_K^P := 
\Cl_K/\langle \cl_K(P)\rangle \to 1. $$

\footnotesize
\begin{verbatim}
{P=x^2+3;bp=2;Bp=10^8;K=bnfinit(P,1);print("P=",P);
forprime(p=bp,Bp,H=bnflog(K,p);if(H!=[[],[],[]],print("p=",p,"  ",H)))}

P=x^2 + 3
p=13       [[13], [13], []]
p=181      [[181], [181], []]
p=2521     [[2521], [2521], []]
p=76543    [[76543], [76543], []]
p=489061   [[489061], [489061], []]
p=6811741  [[6811741], [6811741], []]

P=x^2 + 5
p=5881     [[5881], [5881], []]
\end{verbatim}
\normalsize

These are the only solutions for $p < 10^8$.
More computations would give heuristics to see if the analogous
conjecture: ``$\wt \Cl_K=1$ for all $p \gg 0$'', is credible or not
since the rarefaction of non-trivial cases is similar to that of the 
groups ${\mathcal T}_{K,P}$.

\smallskip
The case of real quadratic fields is studied in \cite[\S\,5.2]{Gr6}\,[Gr2017a] with a table
and in \cite[\S\,2.4]{Ja4}\,[Jau2018a], after the work of Ozaki--Taya \cite{OT}\,[OT1995] and others.

\smallskip
In another direction, the paper \cite{MR1}\,[MR2019a] of Maire--Rougnant gives examples 
of triviality of isotopic components of the torsion groups ${\mathcal T}_{K,P}$;
more precisely the fields $K$ are cyclic extensions of $\Q$ of degrees $3$ and $4$
(from polynomials of Balady, Lecacheux, Balady--Washington) and $S_3$-extensions
of $\Q$.

\smallskip
In \cite{Gr10}\,[Gr2019a], are given numerous programs to test some heuristics and conjectures
about the order of magnitude of the groups ${\mathcal T}_{K,P}$ in totally real
number fields in a Brauer--Siegel framework.

\subsection{Conclusion and open questions}
\smallskip
 In all the aspects of $p$-rationality that we have developed
(theoretical and computational), some interesting applications are done today, including for
instance, for the most recent ones, \cite{HM1}\,[HM2018b] by Hajir--Maire on the $\mu$-invariant
in Iwasawa's theory, then \cite{HMR1}\,[HMR2019a] by Hajir--Maire--Ramakrishna, showing the 
existence of $p$-rational fields having large $p$-rank of the class group, or \cite{HMR2}\,[HMR2019b]
about the existence of a solvable number field $L$, $P$-ramified, whose 
$p$-Hilbert class field tower is infinite. See the bibliographies of these articles to 
expand the list of contributions.

\smallskip
Of course it is not possible to evoque all the studied families of pro-$p$-groups 
having some logical links with $S$-ramification (with more general sets $S$
regarding $P$) as, for instance, that of ``mild groups'' introduced by 
Labute \cite{Lab}\,[Lab2006] (and \cite{LM}\,[LM2011] for the case $p=2$)
dealing with the numbers of generators $d(G)$ and
of relations $r(G)$ and defined as follows: 

\smallskip
{\it A class of finitely presented pro-$p$-groups $G$ of 
cohomological dimension $2$ such that $r(G)\geq d(G)$ and $d(G) \geq 2$ 
arbitrary}. 

\smallskip
Many articles where then published giving an overview of
the wide variety of such groups as the following short excerpt of a result
of Schmidt about global fields \cite[Theorem 1.1]{Sch}\,[Sch2010]:

\smallskip
{\it Let $S, T, {\mathcal M}$ be pairwise disjoint sets of places of $K$, where $S$ and $T$ 
are finite and ${\mathcal M}$ has Dirichlet density $0$. Then there exists a  finite set of places
$S_0$ of $K$ which is disjoint from $S \,\cup\, T \,\cup\, {\mathcal M}$ and such that the 
group ${\mathcal G}_{K,S\,\cup\, S_0}^T$ has cohomological dimension $2$.}

\medskip
But let's go back to the basic abelian invariants, asking some open questions:

\medskip
\quad (i) We know the fixed point formula in a $p$-extension $L/K$ (under 
the conjecture of Leopoldt), but, even in a $p$-cyclic extension with Galois group $G$,
and contrary to the case of $p$-class groups
(as done in \cite{Gr45}\,[Gr2017b] after a very long history), we do not know how to 
compute the filtration 
$(M_i)_{i \geq 0}$, of $M:={\mathcal T}_{K,P}$, defined inductively by:
$$\hbox{$M_0=1$ and $M_{i+1}/M_i := (M/M_i)^G$, for all $i\geq 0$.} $$

\quad (ii) The explicit computation of the $p$-rank, $\wt r_{K,S}^{}$ \eqref{rank}, 
of ${\mathcal A}_{K,S}/{\mathcal T}_{K,S}$ for $S \subseteq P$, is available 
only in favorable Galois cases with an algebraic reasoning on the
canonical representation $\Q_p {\rm log}_S^{}(E_K)$ given by the Herbrand 
theorem on units under Leopoldt's conjecture (see \S\,\ref{comprtild}).

\smallskip
\quad (iii) In the definition of ${\mathcal W}_{K,S}:= 
W_{K,S} /{\rm tor}_{\Z_p}^{}(\ov {E\,}^S_{\!\!K})$,
we do not know how to compute ${\rm tor}_{\Z_p}^{}(\ov {E\,}^S_{\!\!K}) 
\supseteq \iota_S^{}(\mu_K)$ when $S \subsetneq P$.
We ignore, in a $p$-adic framework, if Leopoldt's conjecture is
sufficient to obtain the responses apart from a Galois context. 

\smallskip
A reasonable conjecture is that ${\rm tor}_{\Z_p}^{}(\ov {E\,}^S_{\!\!K}) = \iota_S^{}(\mu_K)$
whatever $K$, $p$ and $S$; but this must be deepened.

\medskip
We hope that our programs in \S\,\ref{PP} may help to give heuristics about this.

\section*{Note}\label{note}
In the programs in verbatim text, one must replace the symbol
of power (in a\^{}b) by the corresponding PARI/GP symbol 
(which is nothing else than that of the computer keyboard); otherwise the 
program does not work (this is due to the character font used by some Journals).
The good print for the programs is also available at:

\url{https://www.dropbox.com/s/1srmksbr2ujf40i/Incomplete%20p-ramification.pdf?dl=0}

\section*{Acknowledgments} I would like to thank Christian Maire and 
Jean-Fran\-\c cois Jaulent for fruitful discussions and information concerning 
some aspects of pro-$p$-groups and $S$-ramification.


\begin{thebibliography}{99}


\bibitem
{Gr3} G. Gras, \emph{Class Field Theory: from theory to practice}, 
corr. 2nd ed., Springer Monographs in Mathematics, Springer, 2005, xiii+507 pages.

\bibitem
{Se1} J-P. Serre, \emph{Cohomologie galoisienne}, Lect. Notes in Math. 5, 
Springer-Verlag 1964, cinqui\`eme \'edition 1991; English translation: \emph{Galois 
cohomology}, Springer 1997; corrected second printing:
Springer Monographs in Math. 2002.

\bibitem
{Sa} I.R. \v Safarevi\v c,  \emph{Extensions with given points of ramification},
Publ. Math. Inst. Hautes Etudes Sci. 18 (1964), 71--95; American Math. Soc.
Transl., Ser. 2, {\bf 59} (1966), 128--149. \par
\url{http://www.numdam.org/article/PMIHES_1963__18__93_0.pdf}

\bibitem
{Br} A. Brumer,  \emph{ Galois groups of extensions of algebraic 
 number fields with given ramification}, Michigan Math. J. {\bf 13} (1966), 33--40. 
\url{https://projecteuclid.org/euclid.mmj/1028999477}

\bibitem
{NSW} J. Neukirch, A. Schmidt, K. Wingberg, 
\emph{Cohomology of Number Fields}, Springer-Verlag (2000), 2nd edition, 
Grundlehren der Math. Wissenschaften 323, Springer-Verlag (2008).

\bibitem
{Win1} K. Wingberg,  \emph{On Galois groups of $p$-closed algebraic number 
fields with restricted ramification}, J. Reine Angew. Math. {\bf 400} (1989), 185--202.
\url{https://eudml.org/doc/153168}

\bibitem
{Win2} K. Wingberg,  \emph{On Galois groups of $p$-closed algebraic number 
fields with restricted ramification II}, J. Reine Angew. Math. {\bf 416} (1991), 187--194.
\url{https://doi.org/10.1515/crll.1991.416.187} 

\bibitem
{Ya} M. Yamagishi,  \emph{A note on free pro-$p$-extensions of algebraic 
number fields}, Journal de th\'eorie des nombres de 
Bordeaux {\bf 5}(1) (1993), 165--178. 
\url{http://www.numdam.org/item/JTNB_1993__5_1_165_0/}

\bibitem
{M01} C. Maire, \emph{On the $\Z_\ell$-rank of abelian extensions 
with restricted ramification}, Journal of Number Theory {\bf 92} (2002), 376--404.
\url{https://doi:10.1006/jnth.2001.2712}

\bibitem
{M02} C. Maire, \emph{On the $\Z_\ell$-rank of abelian extensions 
with restricted ramification (addendum)}, Journal of Number Theory {\bf 98} 
(2003), 217--220. 
\url{https://doi.org/10.1016/S0022-314X(02)00028-8}

\bibitem
{M1} C. Maire, \emph{Sur la dimension cohomologique des 
pro-$p$-extensions des corps de nombres}, J. Th\'eor. Nombres Bordeaux 
{\bf 17}(2) (2005), 575--606.
\url{http://www.numdam.org/item/JTNB_2005__17_2_575_0/}

\bibitem
{Lab} J. Labute, \emph{Mild pro-$p$-groups and Galois groups of 
$p$-extensions of $\Q$}, J. Reine Angew. Math. {\bf 596} (2006), 155--182. \par
\url{https://doi.org/10.1515/CRELLE.2006.058}

\bibitem  
{LM} J. Labute,  J. Min\'a\v{c}, \emph{Mild pro-$2$-groups and $2$-extensions of $\Q$ 
with restricted ramification}, Journal of Algebra {\bf 332}(1) (2011), 136--158.
\url{https://doi.org/10.1016/j.jalgebra.2011.01.019}

\bibitem
{Vog} D. Vogel,  \emph{$p$-extensions with restricted ramification -- The mixed case}
 (2007).\par \url{https://www.mathi.uni-heidelberg.de/~vogel/mixed}

\bibitem
{Neum0} O. Neumann,  \emph{On $p$-closed number fields and an 
analogue of Riemann's existence theorem}. Algebraic number fields: 
$L$-functions and Galois properties, Proc. Sympos., Univ. Durham (1975), 
pp. 625--647. Academic Press, London, 1977.

\bibitem
{Ng2} T. Nguyen Quang Do, \emph{ Sur la $\Z_p$-torsion de certains 
modules galoisiens}, Ann. Inst. Fourier {\bf 36}(2) (1986), 27--46. \par
\url{https://doi.org/10.5802/aif.1045} 

\bibitem
{Ja2} J-F. Jaulent, \emph{Th\'eorie $\ell$-adique globale du  corps de classes}, 
J. Th\'eorie des Nombres de Bordeaux {\bf 10}(2) (1998), 355--397.\par
\url{http://www.numdam.org/article/JTNB_1998__10_2_355_0.pdf}

\bibitem
{BP} F. Bertrandias, J-J. Payan,  \emph{ $\Gamma$-extensions et invariants
cyclotomiques}, Ann. Sci. Ec. Norm. Sup. 4e s\'erie, {\bf 5}(4) (1972), 517--548.
\url{https://doi.org/10.24033/asens.1236}

\bibitem
{Gr7}  G. Gras, \emph{The $p$-adic Kummer-Leopoldt Constant:
Normalized $p$-adic Regulator}, Int. J. Number Theory {\bf 14}(2) (2018), 329--337.  
\url{https://doi.org/10.1142/S1793042118500203}

\bibitem
{Gr8}  G. Gras, \emph{On $p$-rationality of number fields. 
Applications--PARI/GP programs}, Pub. Math. Besan\c con (Th\'eorie des Nombres), 
Ann\'ees 2017/2018.
\url{https://arxiv.org/pdf/1709.06388}

\bibitem
{FV} I.B. Fesenko, S. V. Vostokov, \emph{Local Fields and Their 
Extensions}, American Math Society, Translations of Math Monographs {\bf 121}, 
Second Edition 2002.
\url{https://www.maths.nottingham.ac.uk/personal/ibf/book/vol.pdf}

\bibitem
{Ja00} J-F. Jaulent, \emph{Sur l'ind\'ependance $\ell$-adique de 
nombres alg\'ebriques}, J. Number Theory {\bf 20} (1985), 149--158. \par
\url{https://doi.org/10.1016/0022-314X(85)90035-6}

\bibitem
{Nel} D. Nelson,  \emph{A variation on Leopoldt's conjecture: some
local units instead of all local units} (2013). arxiv:1308.4637 [math.NT] 
\url{https://arxiv.org/pdf/1308.4637}

\bibitem
{P} The PARI Group --\emph{PARI/GP, version \texttt{2.9.0}}, 
Universit\'e de Bordeaux (2016).\par
\url{http://pari.math.u-bordeaux.fr/}

\bibitem
{Mo1} A. Movahhedi, \emph{Sur les $p$-extensions des corps $p$-rationnels}, 
 Th\`ese, Univ. Paris VII (1988). \par
\url{http://www.unilim.fr/pages_perso/chazad.movahhedi/These_1988.pdf}

\bibitem
{Mo2} A. Movahhedi, \emph{Sur les $p$-extensions des corps $p$-rationnels}, 
Math. Nachr. {\bf 149} (1990), 163--176. \par
\url{http://onlinelibrary.wiley.com/doi/10.1002/mana.19901490113/}

\bibitem
{JS1} J-F. Jaulent, O. Sauzet, \emph{Pro-$\ell$-extensions 
de corps de nombres $\ell$-rationnels}, J. Number Th. {\bf 65} (1997), 240--267; 
ibid. 80 (2000), 318--319. \url{https://doi.org/10.1006/jnth.1997.2158}

\bibitem
{JS2} J-F. Jaulent, O. Sauzet, \emph{Extensions quadratiques 
$2$-birationnelles de corps de nombres totalement r\'eels}, Pub. Matem\`atiques 
{\bf 44} (2000), 343--351. 
\url{https://www.math.u-bordeaux.fr/~jjaulent/Articles/Ext2bi.pdf}

\bibitem
{BJ}  C. Bourbon, J-F. Jaulent, \emph{Propagation de la 
$2$-birationalit\'e}, Acta Arithmetica {\bf 160} (2013), 285--301. \par
\url{https://doi.org/10.4064/aa160-3-5}

\bibitem
{Gr10}  G. Gras, \emph{Heuristics and conjectures in the
direction of a $p$-adic Brauer--Siegel theorem}, Math. Comp. {\bf 88}(318) (2019),
1929--1965. \url{https://doi.org/10.1090/mcom/3395}

\bibitem
{Le} E. Lecouturier, \emph{On the Galois structure of the class group
of certain Kummer extensions}, J. London Math. Soc. (2) {\bf 98} (2018), 35--58.
\url{https://doi:10.1112/jlms.12123}

\bibitem
{CL} J. Coates, Y.  Li,  \emph{Non-vanishing theorems for central 
$L$-values of some elliptic curves with complex multiplication} (2018) (in press).
\url{https://arxiv.org/pdf/1811.07595}

\bibitem
{CL2} J. Coates, Y. Li,  \emph{Non-vanishing theorems for central 
$L$-values of some ellip\-tic curves with complex multiplication II} (2019) (in press).
\url{https://arxiv.org/pdf/1904.05756}

\bibitem
{Ng1} T. Nguyen Quang Do, \emph{ Sur la structure galoisienne 
des corps locaux et la th\'eorie d'Iwasawa}, Compositio Mathematica
{\bf 46}(1) (1982), 85--119.
\url{http://www.numdam.org/item/?id=CM_1982__46_1_85_0}

\bibitem
{Ko} H. Koch, \emph{Galois theory of $p$-extensions}
(English translation of \emph{``Galoissche Theorie der $p$-Erweiterungen''}, 1970),
Springer Monographs in Math., Springer 2002.

\bibitem
{HM0}  F. Hajir, C. Maire, \emph{Tamely ramified towers and 
discriminant bounds for number fields}, Compositio Math. {\bf 128}(1) (2001), 35--53.
\url{https://doi.org/10.1023/A:1017537415688}

\bibitem
{HM00} F. Hajir, C. Maire, \emph{ Tamely ramified towers and 
discriminant bounds for number fields II}, Journal of Symbolic 
Computation {\bf 33}(4) (2002), 415--423. 
\url{https://doi.org/10.1006/jsco.2001.0514}

\bibitem
{HM3} F. Hajir, C. Maire, \emph{Extensions of number 
fields with wild ramification of bounded depth},
International Mathematics Research Notices {\bf 13} (2002), 667--696.
\url{http://people.math.umass.edu/~hajir/hajir-imrn.pdf}

\bibitem
{M2} C. Maire, \emph{Cohomology of number fields and analytic 
pro-$p$-groups}, Mosc. Math. J. {\bf 10}(2) (2010), 399--414, 479. \par
\url{http://www.ams.org/distribution/mmj/vol10-2-2010/maire.pdf}

\bibitem
{M3} C. Maire, \emph{On the quotients of the maximal unramified 
$2$-extension of a number field}, Documenta Mathematica {\bf 23} (2018), 1263--1290.
\url{https://doi.org/10.25537/dm.2018v23.1263-1290}

\bibitem
{HM2} F. Hajir, C. Maire, \emph{Analytic Lie extensions 
of number fields with cyclic fixed points and tame ramification} (2018) (in press). \par
\url{https://arxiv.org/pdf/1710.09214}

\bibitem
{HM1} F. Hajir, C. Maire, \emph{Prime decomposition and the 
Iwasawa mu-invariant}, Math. Proc. Camb. Phil. Soc. {\bf 166} (2019), 599--617.
Published online: 26 April 2018.
\url{https://doi.org/10.1017/S0305004118000191}

\bibitem
{Neum} O. Neumann, \emph{ On $p$-closed algebraic number fields with 
restricted ramification}, Izv. Akad. Nauk USSR, ser. Math. 39, 2 (1975), 259--271; English 
translation: Math. USSR, Izv. {\bf 9} (1976), 243--254.

\bibitem
{Hab} K. Haberland,  \emph{Galois cohomology of algebraic number fields}.
With two appendices by Helmut Koch, Thomas Zink, V.E.B. Deutscher Verlag der 
Wissenschaften 1978.

\bibitem
{Sch} A. Schmidt,  \emph{\"Uber pro-$p$-fundamentalgruppen 
markierter arithmetischer kurven}, J. Reine Angew. Math. {\bf 640} (2010), 203--235.
\url{https://www.mathi.uni-heidelberg.de/~schmidt/papers/marked.pdf}

\bibitem
{ElHZ}  A. El Habibi, M. Ziane,  \emph{$p$-Rational Fields and the 
Structure of Some Modules} (2018) arxiv:1804.10165 [math.NT].
\url{https://arxiv.org/pdf/1804.10165}

\bibitem
{Kub} T. Kubota, \emph{Galois group of the maximal abelian extension of an
algebraic number field}, Nagoya Math. J. {\bf 12} (1957), 177--189.

\bibitem
{Mi} H. Miki, \emph{ On the maximal abelian $\ell$-extension of a finite algebraic
number field with given ramification}, Nagoya Math. J. {\bf 70} (1978), 183--202.
\url{https://doi.org/10.1017/S0027763000021875}

\bibitem
{Gr01}  G. Gras, \emph{ Groupe de Galois de la $p$-extension 
ab\'elienne $p$-ramifi\'ee maximale d'un corps de nombres}, J. reine angew. 
Math. {\bf 333} (1982), 86--132. \par \url{https://eudml.org/doc/152440}
\url{https://www.researchgate.net/publication/243110955} 

\bibitem
{Gr2}  G. Gras, \emph{Th\'eor\`emes de r\'eflexion},
J. Th\'eor. Nombres Bordeaux {\bf 10}(2) (1998), 399--499. \par
\url{http://www.numdam.org/item/JTNB_1998__10_2_399_0/}

\bibitem
{Gr11}  G. Gras, \emph{Test of Vandiver's conjecture with 
Gauss sums--Heuristics} (2019) arxiv:1808.03443 [math.NT] \par
\url{https://arxiv.org/pdf/1808.03443}

\bibitem
{EV} J.S. Ellenberg,  A. Venkatesh,  \emph{Reflection principles 
and bounds for class group torsion}, Int. Math. Res. Not. {\bf 2007}(1) (2007).
\url{https://doi.org/10.1093/imrn/rnm002}

\bibitem
{Gr12}  G. Gras, \emph{Annihilation of ${\rm tor}_{\Z_p}
({\mathcal G}_{K,S}^{\rm ab})$ for real abelian extensions $K/\Q$}, 
Communications in Advanced Mathematical Sciences {\bf 1}(1) (2018), 5--34.
\url{http://dergipark.gov.tr/download/article-file/543993}

\bibitem
{Or} B. Oriat,  \emph{Lien alg\'ebrique entre les deux facteurs 
de la formule analytique du nombre de classes dans les corps ab\'eliens},
Acta Arithmetica {\bf 46} (1986), 331--354.
\url{https://doi.org/10.4064/aa-46-4-331-354}

\bibitem
{AF} Y.  Amice, J. Fresnel,  \emph{Fonctions z\^eta $p$-adiques des corps
de nombres ab\'eliens r\'eels}, Acta Arithmetica {\bf 20}(4) (1972),  353--384.
\url{http://matwbn.icm.edu.pl/ksiazki/aa/aa20/aa2043.pdf}

\bibitem
{Co} J. Coates,  \emph{$p$-adic $L$-functions and Iwasawa's theory}, 
In: Proc. of Durham Symposium 1975, New York-London (1977), 269--353. 

\bibitem
{Se2}  J-P. Serre, \emph{Sur le r\'esidu de la fonction z\^eta $p$-adique 
d'un corps de nombres}, C.R. Acad. Sci. Paris {\bf 287} (1978), S\'erie I, 183--188.

\bibitem
{Col} P. Colmez, \emph{ R\'esidu en $s = 1$ des fonctions z\^eta $p$-adiques},
Invent. Math. {\bf 91} (1988), 371--389. \par
\url{https://eudml.org/doc/143545} 

\bibitem
{AM}  J. Assim,  A. Movahhedi, \emph{Galois codescent for
motivic tame kernels} (2019) (in press).\par
\url{https://arxiv.org/pdf/1901.07219}

\bibitem
{Gr00}  G. Gras, \emph{\'Etude d'invariants relatifs aux groupes des classes 
des corps ab\'eliens}, Journ\'ees Arithm\'etiques de Caen (Univ. Caen, Caen, 1976), 
pp. 35--53. Ast\'e\-risque No. 41--42, Soc. Math. France, Paris (1977). \par
\url{http://www.numdam.org/book-part/AST_1977__41-42__35_0/}

\bibitem
{Ri} K. A. Ribet,  Bernoulli numbers and ideal classes,
\emph{Gaz. Math.} {\bf 118} (2008), 42--49. \par
\url{http://smf4.emath.fr/Publications/Gazette/2008/118/smf_gazette_118_42-49.pdf}

\bibitem
{Gr02}  G. Gras, \emph{ Logarithme $p$-adique et groupes 
de Galois}, J. reine angew. Math. {\bf 343} (1983), 64--80. \par
\url{https://doi.org/10.1515/crll.1983.343.64}

\bibitem
{Gr03}  G. Gras, \emph{Sur la $p$-ramification ab\'elienne}, 
Conf\'erence donn\'ee \`a l'University Laval, Qu\'ebec, Mathematical series 
of the department of mathematics {\bf 20} (1984), 1--26. \par available at
\url{https://www.dropbox.com/s/fusia63znk0kcky/Lectures1982.pdf?dl=0}

\bibitem
{Gr1}  G. Gras, \emph{Remarks on $K_2$ of number fields},
Jour. Number Theory {\bf 23}(3) (1986), 322--335. \par
\url{http://www.sciencedirect.com/science/article/pii/0022314X86900776}

\bibitem
{Gr0}  G. Gras, \emph{Plongements kumm\'eriens dans les 
$\Z_p$-extensions}, Compositio Mathematica {\bf 55}(3) (1985), 383--396.\par
\url{http://www.numdam.org/item/?id=CM_1985__55_3_383_0}

\bibitem
{Ja1} J-F. Jaulent,  L'arithm\'etique des $\ell$-extensions 
(Th\`ese de doctorat d'Etat), Pub. Math. Besan\c con (Th\'eorie des Nombres)
(1986), 1--349.
\url{http://pmb.univ-fcomte.fr/1986/Jaulent_these.pdf}

\bibitem
{Che} C. Chevalley,  \emph{Sur la th\'eorie du corps de classes 
dans les corps finis et les corps locaux} (Th\`ese no. 155), Jour. of the Faculty 
of Sciences Tokyo {\bf 2} (1933), 365--476. 
\url{http://archive.numdam.org/item/THESE_1934__155__365_0/}

\bibitem
{Ja0} J-F. Jaulent, \emph{$S$-classes infinit\'esimales d'un corps 
de nombres alg\'ebriques}, Ann. Sci. Inst. Fourier {\bf 34}(2) (1984), 1--27. \par
\url{https://doi.org/10.5802/aif.960}

\bibitem
{MN} A. Movahhedi, T. Nguyen Quang Do, \emph{Sur l'arithm\'etique 
des corps de nombres $p$-rationnels}, S\'eminaire de Th\'eorie des Nombres, 
Paris 1987--88, Progress in Math. {\bf 81} (1990), 155--200. \par
\url{https://doi.org/10.1007/978-1-4612-3460-9_9}

\bibitem
{GJ}  G. Gras,  J-F. Jaulent,  \emph{Sur les corps de nombres r\'eguliers}, 
Math. Z. {\bf 202}(3) (1989), 343--365. \par
\url{https://eudml.org/doc/174095}

\bibitem
{JN} J-F. Jaulent, T. Nguyen Quang Do, \emph{Corps $p$-rationnels, 
corps $p$-r\'eguliers et ramification restreinte}, J. Th\'eor. Nombres 
Bordeaux {\bf 5} (1993), 343--363. 
\url{http://www.numdam.org/article/JTNB_1993__5_2_343_0.pdf}

\bibitem
{Ja11} J-F. Jaulent, \emph{Classes logarithmiques des corps 
de nombres}, J. Th\'eorie des Nombres de Bordeaux {\bf 6} (1994), 301--325. \par
\url{http://www.numdam.org/article/JTNB_1994__6_2_301_0.pdf}

\bibitem
{Ja3} J-F. Jaulent, \emph{Classes logarithmiques des corps 
totalement r\'eels}, Acta Arithmetica {\bf 103} (2002), 1--7.\par
\url{https://www.math.u-bordeaux.fr/~jjaulent/Articles/CLogTR.pdf}

\bibitem
{FGS} L.J. Federer,  B. H. Gross,  (with an appendix by W. Sinnot),
\emph{Regulators and Iwasawa Modules}, Invent. Math.  {\bf 62}(3) (1981), 443--457.

\bibitem
{Kuz} L.V. Kuz'min,  \emph{The Tate module of algebraic number fields},
Izv. Akad. Nauk SSSR {\bf 36} (1972), 267--327. \par
\url{http://mi.mathnet.ru/eng/izv/v36/i2/p267}

\bibitem
{Ja10} J-F. Jaulent, \emph{Sur les conjectures de Leopoldt et 
de Gross}, Actes des Journ\'ees Arithm\'etiques de Besan\c con (1985), 
Ast\'erisque {\bf 147/148} (1987), 107--120.
\url{http://www.numdam.org/item/AST_1987__147-148__107_0/}

\bibitem
{Ja35} J-F. Jaulent, \emph{ Sur les normes cyclotomiques et les conjectures 
de Leopoldt et de Gross-Kuz'min}, Annales. Math. Qu\'ebec {\bf 41} (2017), 119--140.
\url{https://doi.org/10.1007/s40316-016-0069-3}

\bibitem
{Ja33} J-F. Jaulent, \emph{Classes logarithmiques et capitulation},
 Functiones et Approximatio  {\bf 54}(2) (2016), 227--239. 
 \url{https://projecteuclid.org/euclid.facm/1466450668}

\bibitem
{Ja34} J-F. Jaulent, \emph{ Sur la capitulation pour le module de Bertrandias-Payan},
Pub. Math. Besan\c con (Th\'eorie des Nombres) (2016), 45--58.  
\url{http://pmb.univ-fcomte.fr/2016/Jaulent.pdf}

\bibitem
{Ja55} J-F. Jaulent, \emph{Principalisation ab\'elienne des 
groupes de classes de rayons} (2018) (in press).\par
\url{https://arxiv.org/pdf/1801.07173}

\bibitem
{Ja44} J-F. Jaulent, \emph{ Principalisation des groupes de classes logarithmiques}
 (2018). To appear in Functiones et Approximatio.
\url{https://arxiv.org/pdf/1801.07176}

\bibitem
{BG}  R.I. Berger,  G. Gras, \emph{ Regular fields: normic 
criteria in $p$-extensions}, Pub. Math. Besan\c con (Th\'eorie des Nombres),
Ann\'ees 1991/92.
\url{http://pmb.univ-fcomte.fr/1992/Berger_Gras.pdf}

\bibitem
{Ng3} T. Nguyen Quang Do, \emph{Lois de r\'eciprocit\'e primitives},
Manuscripta math. {\bf 72}(1) (1991), 307--324. \par
\url{https://doi.org/10.1007/BF02568282}

\bibitem
{Ja100} J-F. Jaulent, \emph{La th\'eorie de Kummer et le ${\rm K}_2$ des corps de nombres}, 
J. Th\'eorie des Nombres de Bordeaux {\bf 2} (1990), 377-411.  
\url{http://www.numdam.org/item/JTNB_1990__2_2_377_0/}

\bibitem
{BJ0} K. Belabas,  J-F. Jaulent, \emph{The logarithmic class group 
package in PARI/GP}, Pub. Math. Besan\c con (Th\'eorie des Nombres) (2016), 5--18.\par
\url{http://pmb.univ-fcomte.fr/2016/pmb_2016.pdf}

\bibitem
{Ng33} T. Nguyen Quang Do, \emph{Analogues sup\'erieurs du noyau sauvage},
J. Th\'eor. Nombres Bordeaux {\bf 2}(4) (1991), 263--271. \par
\url{http://www.numdam.org/item/?id=JTNB_1992__4_2_263_0}

\bibitem
{JM} J-F. Jaulent, A. Michel, \emph{Approche logarithmique des noyaux \'etales des 
corps de nombres}, J. Number Th.  {\bf 120} (2006), 72--91.
\url{https://doi.org/10.1016/j.jnt.2005.11.011}

\bibitem
{Hut} K. Hutchinson, \emph{Tate kernels, \'etale ${\rm K}$-theory and the Gross kernel} (2017)
arxiv:1709.06465 [math.NT].
\url{https://arxiv.org/pdf/1709.06465}

\bibitem
{HMR1} F. Hajir, C. Maire, R. Ramakrishna,  \emph{Cutting towers 
of number fields} (2019) arxiv:1901.04354 [math.NT]. \par
\url{https://arxiv.org/pdf/1901.04354}

\bibitem
{HMR2} F. Hajir, C. Maire, R. Ramakrishna,  \emph{Infinite 
class field towers of number fields of prime power discriminant} (2019)
arxiv:1904.07062 [math.NT]. \par
\url{https://arxiv.org/pdf/1904.07062}

\bibitem
{AS}  A. Angelakis, P. Stevenhagen,  \emph{ Absolute abelian Galois 
groups of imaginary quadratic fields}, In: proceedings volume of ANTS-X, 
UC San Diego 2012, OBS 1 (2013).
\url{http://msp.org/obs/2013/1-1/obs-v1-n1-p02-p.pdf}

\bibitem
{On} M. Onabe, \emph{On the isomorphisms of the Galois groups of the 
maximal abelian extensions of imaginary quadratic fields}, Natur. Sci. Rep. 
Ochanomizu Univ. {\bf 27}(2) (1976), 155--161. \par
\url{https://www.researchgate.net/publication/37823146}

\bibitem
{Gr4}  G. Gras,  \emph{On the structure of the Galois group of the 
Abelian closure of a number field}, J. Th\'eorie Nombres 
Bordeaux {\bf 26}(3) (2014), 635--654.
\url{http://www.numdam.org/article/JTNB_2014__26_3_635_0.pdf}

\bibitem
{Gr5}  G. Gras, \emph{Les $\theta$-r\'egulateurs locaux d'un 
nombre alg\'ebrique : Conjectures $p$-adiques}, Canadian Journal of 
Mathematics {\bf 68}(3) (2016), 571--624. \par
\url{http://dx.doi.org/10.4153/CJM-2015-026-3}
\url{https://arxiv.org/pdf/1701.02618}

\bibitem
{Go} E.Z. Goren,  \emph{Hasse invariants for Hilbert modular varieties},
Isr. J. Math. {\bf 122} (2001), 157--174. \par
\url{https://link.springer.com/article/10.1007/BF02809897}

\bibitem
{BGKK} G. Boeckle, D.-A. Guiraud, S. Kalyanswamy, C. Khare, 
\emph{Wieferich Primes and a mod $p$ Leopoldt Conjecture} (2018)
arxiv:1805.00131 [math.NT]. 
\url{https://arxiv.org/pdf/1805.00131}

\bibitem
{Gre1} R. Greenberg, \emph{On the Iwasawa invariants of totally 
real number fields}, Amer. J. Math. {\bf 98}(1) (1976), 263--284. \par
\url{http://www.jstor.org/stable/2373625?}

\bibitem
{Ng6} T. Nguyen Quang Do, \emph{Formules de genres 
et conjecture de Greenberg},  Annales math\'ematiques du Qu\'ebec
{\bf 42}(2) (2018), 267--280.
\url{https://doi.org/10.1007/s40316-017-0093-y}

\bibitem
{Gr6}  G. Gras, \emph{Approche $p$-adique de la conjecture de 
Greenberg pour les corps totalement r\'eels}, Ann. Math. Blaise Pascal {\bf 24}(2) 
(2017), 235--291.
\url{http://ambp.cedram.org/item?id=AMBP_2017__24_2_235_0}

\bibitem
{Gr61}  G. Gras, \emph{Normes d'id\'eaux dans la tour cyclotomique
et conjecture de Greenberg}, Annales math\'ematiques du Qu\'ebec,
Online: 17 October 2018, 1--32.
\url{https://doi.org/10.1007/s40316-018-0108-3}

\bibitem
{Ja6} J-F. Jaulent, \emph{Normes universelles et conjecture de Greenberg}
(2019) arxiv:1904.07014 [math.NT]. \par
\url{https://arxiv.org/pdf/1904.07014}

\bibitem
{Ja4} J-F. Jaulent, \emph{ Note sur la conjecture de Greenberg},
J. Ramanujan Math. Soc. {\bf 34}(1) (2019), 59--80. \par
\url{http://www.mathjournals.org/jrms/2019-034-001/2019-034-001-005.html} \par
\url{https://arxiv.org/pdf/1612.00718}

\bibitem
{Gr111}  G. Gras, \emph{Principalisation d'id\'eaux par extensions 
absolument ab\'eliennes}, J. Number Th. {\bf 62}(2) (1997), 403--421. \par
\url{https://doi.org/10.1006/jnth.1997.2068}

\bibitem
{Bos} S. Bosca,  \emph{Principalization of ideals in abelian extensions
of number fields}, International Journal of Number Theory {\bf 5}(03) (2009), 527--539.
\url{https://doi.org/10.1142/S1793042109002213}

\bibitem
{Ng4} T. Nguyen Quang Do, \emph{Sur la conjecture faible de 
Greenberg dans le cas ab\'elien $p$-d\'ecompos\'e}, Int. J. Number
Theory {\bf 2} (2006), 49--64.

\bibitem
{Ng5} T. Nguyen Quang Do, \emph{Sur une forme faible 
de la conjecture de Greenberg II}, Int. J. Number Theory {\bf 13}
(2017), 1061--1070.
\url{https://doi.org/10.1142/S1793042117500567}

\bibitem
{Gre2} R. Greenberg, \emph{Galois representations with open image},
Annales de Math\'e\-matiques du Qu\'ebec, special volume in honor of Glenn 
Stevens, {\bf 40}(1) (2016), 83--119. \par
\url{https://link.springer.com/article/10.1007/s40316-015-0050-6}

\bibitem
{Pi} F. Pitoun,  \emph{Calculs th\'eoriques et explicites en th\'eorie d'Iwasawa},
Th\`ese de doctorat en Math\'ematiques, Laboratoire de Math\'ematiques,
Universit\'e de Franche-comt\'e Besan\c con (2010).
\url{https://www.theses.fr/220448329}

\bibitem
{PV} F. Pitoun, F. Varescon,  \emph{Computing the torsion 
of the $p$-ramified module of a number field}, Math. Comp. 
{\bf 84}(291) (2015), 371--383.
\url{https://doi.org/10.1090/S0025-5718-2014-02838-X }

\bibitem
{Si} J.H. Silverman,  \emph{ Wieferich's criterion and the $abc$-conjecture}, 
Journal of Number Theory {\bf 30} (1988), 226--237. \par
\url{https://doi.org/10.1016/0022-314X(88)90019-4}

\bibitem
{GM} H. Graves, M.R.  Murty, \emph{The $abc$ conjecture and 
non-Wieferich primes in arithmetic progressions}, Journal of Number 
Theory {\bf 133}(6) (2013), 1809--1813.\par
\url{http://www.sciencedirect.com/science/article/pii/S0022314X12003368}

\bibitem
{MR2}  C. Maire, M. Rougnant,  \emph{A note on $p$-rational fields 
and the $abc$-conjecture} (2019) arxiv:1903.11271 [math.NT].\par
\url{https://arxiv.org/pdf/1903.11271}

\bibitem
{By} D. Byeon, \emph{Indivisibility of special values of Dedekind zeta functions of real quadratic
fields}, Acta Arithmetica  {\bf 109}(3) (2003), 231--235.

\bibitem
{AB} J. Assim, Z. Bouazzaoui, \emph{Half-integral weight modular forms and real quadratic 
$p$-rational fields} (2018) [math.NT].
\url{https://arxiv.org/pdf/1906.03344}

\bibitem
{Shu} J. Shu,  \emph{Root numbers and Selmer groups for the Jacobian
varieties of Fermat curves} (2018) arxiv:1809.09285 [math.NT]. \par
\url{https://arxiv.org/pdf/1809.09285}

\bibitem
{DP} R. Davis, R. Pries,  \emph{Cohomology groups of Fermat 
curves via ray class fields of cyclotomic fields} (2018) (in press). \par
\url{https://arxiv.org/pdf/1806.08352}.

\bibitem
{Cha}  A. Charifi, \emph{Groupes de torsion
attach\'es aux extensions Ab\'eliennes $p$-ramifi\'ees maximales (cas
des corps totalement r\'eels et des corps quadratiques imaginaires)},
Th\`ese de $3^e$ cycle, Math\'ematiques, Universit\'e de Franche-Comt\'e (1982), 50 pp.

\bibitem
{Hat1} K. Hatada,  \emph{Mod 1 distribution of Fermat and Fibonacci 
quotients and values of zeta functions at $2-p$},  
Comment. Math. Univ. St. Pauli  {\bf 36} (1987), 41--51. 

\bibitem
{Hat2} K. Hatada, \emph{Chi-square tests for mod 1 
distribution of Fermat and Fibonacci quotients},
Sci. Rep. Fac. Educ., Gifu Univ., Nat. Sci. {\bf 12} (1988), 1--2.

\bibitem
{HZ} T. Hofmann, Y.  Zhang, \emph{Valuations of $p$-adic 
regulators of cyclic cubic fields}, Journal of Number Theory {\bf 169} (2016), 86--102.
\url{https://doi.org/10.1016/j.jnt.2016.05.016}

\bibitem
{KS} J.S. Kraft, R. Schoof, \emph{Computing Iwasawa modules 
of real quadratic number fields}, Compositio Math. {\bf 97} (1995), 135--155.
\url{http://www.numdam.org/item/CM_1995__97_1-2_135_0/}

\bibitem
{BR} R. Barbulescu, J. Ray, \emph{Some remarks and 
experimentations on Greenberg's $p$-rationality conjecture} (2017).\par
arxiv:1706.04847 [math.NT]. \url{https://arxiv.org/pdf/1706.04847}

\bibitem
{Bou} Z. Bouazzaoui,  \emph{Fibonacci sequences and real quadratic 
$p$-rational fields} (2019) arxiv:1902.04795 [math.NT].\par
\url{https://arxiv.org/pdf/1902.04795}

\bibitem
{OT}  M. Ozaki,   H. Taya, \emph{A note on Greenberg's conjecture for 
real abelian number fields}, Manuscripta Math. {\bf 88}(1) (1995), 311--320.
\url{http://link.springer.com/article/10.1007/BF02567825}

\bibitem
{MR1} C. Maire, M. Rougnant, \emph{Composantes 
isotypiques de pro-$p$-extensions de corps de nombres et $p$-rationalit\'e},
Publ. Math. Debrecen {\bf 94}(1/2) (2019), 123--155.\par
\url{https://lmb.univ-fcomte.fr/IMG/pdf/maire-rougnant-08_22_2018.pdf}

\bibitem
{Gr45}  G. Gras, \emph{Invariant generalized ideal 
classes--Structure theorems for $p$-class groups in $p$-extensions},
Proc. Indian Acad. Sci. (Math. Sci.) {\bf 127}(1) ( 2017), 1--34. \par
\url{https://www.ias.ac.in/article/fulltext/pmsc/127/01/0001-0034}

\end{thebibliography}
\end{document}